\documentclass[12pt,a4paper]{amsart}

\usepackage{amsmath,amssymb}
\usepackage[margin=2.5cm]{geometry}
\usepackage{hyperref}
\usepackage{xcolor}

\newtheorem{theorem}{Theorem}[section]
\newtheorem{lemma}[theorem]{Lemma}
\newtheorem{proposition}[theorem]{Proposition}
\newtheorem{corollary}[theorem]{Corollary}

\theoremstyle{definition}
\newtheorem{definition}[theorem]{Definition}
\newtheorem{example}[theorem]{Example}

\newtheorem{remark}[theorem]{Remark}

\title[Smallest numerical semigroups closed under affine maps]{On the smallest numerical semigroups closed under affine maps}

\author{A. \'Alvarez}
\address{A. \'Alvarez\\ Departamento de Matem\'aticas, Universidad de Extremadura, 06071 Badajoz, Spain. ORCID: \href{https://orcid.org/0000-0002-9095-3221}{0000-0002-9095-3221}}
\email{aalarma@unex.es}

\author{C.-J. Moreno-\'Avila*}
\address{C.-J. Moreno-\'Avila\\ Departamento de Matem\'aticas, Universidad de Extremadura, 10071 C\'aceres, Spain. ORCID: \href{https://orcid.org/0000-0002-2374-5932}{0000-0002-2374-5932}}
\email{cjmoravi@unex.es}

\author{I. Ojeda}
\address{I. Ojeda\\ Departamento de Matem\'aticas, Universidad de Extremadura, 06071 Badajoz, Spain. ORCID: \href{https://orcid.org/0000-0003-3173-5934}{0000-0003-3173-5934}}
\email{ojedamc@unex.es}
\begin{document}

\begin{abstract}
We study numerical semigroups $S_{a,b}(m)$ generated by the orbit of $m$ under the affine map $T_{a,b}(z)=az+b$, where $a\ge 2$, $m>1$, $\gcd(b,m)=1$, and $b\ge -(a-2)m-2$. This extends the usual affine-closed setting to feasible negative values of $b$. We write $A_i=(a^i-1)/(a-1)$ and let $n$ be the smallest positive integer such that $A_n\ge m$.

We determine the minimal generators and give an explicit description of the Ap\'ery set, obtaining homogeneity and formulas for the Frobenius number and the genus. We also study pseudo-Frobenius numbers via the induced Ap\'ery parametrization and prove the sharp upper bound $\operatorname{t}(S_{a,b}(m))\le n-1$ for the type.

We give a complete characterization of the symmetric members of the family in terms of the canonical representative of $m-1$. Finally, we exhibit a subfamily whose pseudo-Frobenius numbers form an arithmetic progression of length $n-1$; in particular, this subfamily attains the bound.
\end{abstract}



\thanks{The second and third authors are partially supported by projects PID2022-138906NB-C22 and PID2022-138906NB-C21, respectively, funded by MCIN/AEI/10.13039/501100011033 and the European Union NextGenerationEU/PRTR. All three authors are also partially supported by grant GR24068, funded by the Junta de Extremadura and co-funded by the European Regional Development Fund (ERDF)}

\thanks{* Corresponding author}

\maketitle

\section{Introduction}

Throughout the paper, we use the convention $\mathbb{N}=\{0,1,2,\dots\}$. A numerical semigroup is a submonoid $S$ of $(\mathbb{N},+)$ such that $\mathbb{N}\setminus S$ is finite. Equivalently, $S$ is an additive submonoid of the nonnegative integers containing $0$ and generated by a finite set of positive integers with greatest common divisor equal to $1$. A standard reference for the theory of numerical semigroups is \cite{libro}.

Several arithmetic invariants play a central role in the study of numerical semigroups. The multiplicity $\operatorname{m}(S)$ of $S$ is its smallest positive element. By \cite[Theorem~2.7]{libro}, every numerical semigroup has a unique finite minimal system of generators, denoted $\operatorname{msg}(S)$; its cardinality is the embedding dimension $\operatorname{e}(S)$. The Frobenius number $\operatorname{F}(S)$ is the largest integer not belonging to $S$, and the genus is the number of gaps, $g(S)=|\mathbb{N}\setminus S|$. The Ap\'ery set of $S$ with respect to $\operatorname{m}(S)$ is
\[
\operatorname{Ap}(S)=\operatorname{Ap}(S,\operatorname{m}(S))=\{s\in S\mid s-\operatorname{m}(S)\notin S\},
\]
and it encodes much of the structure of $S$. Another important invariant is the type $\operatorname{t}(S)$, defined as the cardinality of the set of pseudo-Frobenius numbers of $S$. Recall that an integer $x\notin S$ is a pseudo-Frobenius number of $S$ if $x+s\in S$ for all $s\in S\setminus\{0\}$.

Given integers $a,b$, let $T_{a,b}:\mathbb Z\to\mathbb Z$ be the affine map $T_{a,b}(z)=az+b$. Following \cite{Ugo17}, a numerical semigroup $S$ is said to be closed under $T_{a,b}$ if $T_{a,b}(S\setminus\{0\})\subseteq S\setminus\{0\}$. Ugolini initiated a systematic study of such semigroups for affine maps with positive parameters, providing a general construction that unifies several classical families, including Thabit, Mersenne, and repunit numerical semigroups. These families had been previously studied in detail by Branco, Rosales, and Torr\~ao; see, for instance, \cite{RosBraTorThabit,RosBraTorMersenne,RosBraTorRepunit}. Numerical semigroups defined by closure under affine maps have also been studied from an algorithmic viewpoint; see Robles-P\'erez and Rosales \cite{RoblesRosales18}, who introduce $(a,b)$-monoids and show that the corresponding families form Frobenius varieties.

Affine-recursive constructions also appear in related work. In \cite{BrancoColacoOjeda23,BrancoColacoOjeda21}, and in subsequent work by Ojeda and Cola\c{c}o \cite{ColacoOjeda25}, such semigroups were studied from structural and algorithmic viewpoints. Independently, in \cite{XIN}, large families of numerical semigroups defined by recursive generating systems were considered, with explicit formulas for invariants such as the Frobenius number, genus, and, in special subfamilies, the type and pseudo-Frobenius numbers. Although approached from a different perspective, many of these constructions overlap conceptually with the affine-closed setting considered here.

The purpose of this paper is to extend the theory of numerical semigroups closed under affine maps to the feasible range where the constant term may be nonpositive. More precisely, we study the smallest numerical semigroup of multiplicity $m$ closed under $T_{a,b}$ when $a\ge 2$, $m>1$, $\gcd(b,m)=1$, and $b\ge -(a-2)m-2$. This includes the classical case $b\ge0$ studied by Ugolini in \cite{Ugo17}, but also allows negative values of $b$, where new phenomena arise. We set
\[
A_i=\frac{a^i-1}{a-1}\qquad(i\ge 0),
\]
and denote by $n$ the smallest positive integer such that $A_n\ge m$.

Our approach is based on an explicit combinatorial description of the Ap\'ery set of $S_{a,b}(m)$. This description yields formulas for the Frobenius number and the genus, implies homogeneity of the semigroups $S_{a,b}(m)$, and allows us to study pseudo-Frobenius numbers uniformly for positive and negative values of $b$. It also leads to a complete characterization of the symmetric members of the family, giving an explicit infinite class of type-one numerical semigroups, equivalently symmetric semigroups, and hence of Gorenstein monomial curves.

The paper is organized as follows. In Section~\ref{Sect:2}, we first determine which affine maps can preserve numerical semigroups and then prove that, under the feasible hypotheses, $S_{a,b}(m)$ is the smallest numerical semigroup of multiplicity $m$ closed under $T_{a,b}$; this is Theorem~\ref{th:1}. For $b>0$, this recovers the affine-closed numerical semigroups studied by Ugolini \cite{Ugo17}, and it also falls into the one-map case of the $(a,b)$-monoids introduced by Robles-P\'erez and Rosales \cite{RoblesRosales18}. The same construction extends here to admissible negative values of $b$.

Section~\ref{Sect:e} is devoted to the embedding dimension. We show that the first $n$ affine iterates always belong to the minimal system of generators, and we give a simple sufficient condition ensuring that there are no further minimal generators; see Proposition~\ref{prop:3} and Corollary~\ref{cor:e=n}. This condition is automatic under the standing hypothesis $b\ge -(a-2)m-2$, as observed in Remark~\ref{rem:e=n}.

In Section~\ref{sec:apery}, we compute the Ap\'ery set explicitly. The main result is Theorem~\ref{Th:Ap}, which describes $\operatorname{Ap}(S)$ in terms of canonical representatives in $R(a,n)$ (see Definition~\ref{def:conj_R}). This leads to the parametrization
\[
\operatorname{Ap}(S)=\{d\ell+\mathsf m_A(\ell)m\mid 0\le \ell\le m-1\}
\]
in Corollary~\ref{cor.OP}. As consequences, we obtain homogeneity in the sense of Jafari and Zarzuela \cite{JafariZarzuela18}, as well as formulas for the Frobenius number and the genus in Corollary~\ref{cor:FrobGen}. We also isolate in Corollary~\ref{cor:F_at_m-1} the precise endpoint condition under which the Frobenius formula from the nonnegative case of Ugolini \cite{Ugo17} remains valid. The explicit description of $\operatorname{Ap}(S)$ can be viewed as an analogue of the shifted-family Ap\'ery description of O'Neill and Pelayo \cite{ONeillPelayo18}.

Section~\ref{sec:PF} studies pseudo-Frobenius numbers using the Ap\'ery parametrization. Proposition~\ref{prop:PF2} gives a first finite set of candidates, while Theorem~\ref{th:maxcrit_mA} provides an intrinsic criterion for deciding when a given element $W(\ell)$ is maximal in the Ap\'ery poset. This criterion implies that all maximal elements lie in two narrow bands determined by the last coordinate of the canonical representative; see Corollary~\ref{cor:two_layers} and Corollary~\ref{cor:two_bands}.

In Section~\ref{sec:symmetric_family}, we characterize the symmetric members of the family. The main result, Theorem~\ref{th:type-one}, says that $S_{a,b}(m)$ is symmetric, equivalently of type one, precisely when either $n=2$ or the canonical representative of $m-1$ has the form
\[
(a,a-1,\ldots,a-1,c)
\]
with $1\le c\le a-1$. This gives an explicit infinite class of type-one numerical semigroups, equivalently symmetric semigroups, and hence of Gorenstein monomial curves; see, for instance, Barucci, Dobbs, and Fontana~\cite{BDF97}, and Gimenez and Srinivasan~\cite{GimenezSrinivasan14}.

Section~\ref{sec:type_bound} proves the sharp upper bound $\operatorname{t}(S)\le n-1$. The proof is purely combinatorial and is based on the Ap\'ery-poset criterion from Section~\ref{sec:PF}, together with an induction on the number of available Ap\'ery weights.

Finally, Section~\ref{sec:arith_pf_family} exhibits a family attaining the upper bound for the type. For $m=cA_{n-1}+1$, with $1\le c\le a$, Theorem~\ref{th:arith_family_PF} proves that
\[
\operatorname{PF}(S)=\{\alpha,\alpha-b,\alpha-2b,\ldots,\alpha-(n-2)b\}
\]
for a suitable integer $\alpha$; in particular, $\operatorname{t}(S)=n-1$. Hence, by Theorem~\ref{thm:type-bound}, this family is extremal. We also explain how it sits inside the Collection CNS of~\cite{XIN}, while not being directly covered by the available results there on pseudo-Frobenius numbers and type. The section ends with a determinantal remark: a natural matrix $X_c$ gives candidates for the defining equations of the associated toric ideal, recovering the generalized repunit case of Branco, Cola\c{c}o, and Ojeda~\cite{BrancoColacoOjeda21} when $c=a$; the corresponding Eagon--Northcott resolution is related to the work of Cola\c{c}o and Ojeda~\cite{ColacoOjeda25}.

Unless otherwise stated, all computations in the examples were performed with the aid of the \texttt{GAP} \cite{GAP} package \texttt{NumericalSgps} \cite{numericalsgps}.

\section{Numerical semigroups closed under affine maps}\label{Sect:2}

Let $T$ be an affine map on the integers, that is, a map
$T:\mathbb{Z}\to\mathbb{Z}$ of the form $T(z)=az+b$ for some
$a,b\in\mathbb{Z}$. We write $T_{a,b}(z)=az+b$.

\begin{definition}
A numerical semigroup $S$ is said to be closed under the affine map
$T_{a,b}$ if $T_{a,b}(S\setminus\{0\})\subseteq S\setminus\{0\}$.
Equivalently, $as+b\in S\setminus\{0\}$ for every
$s\in S\setminus\{0\}$.
\end{definition}

We first record two elementary necessary conditions on the coefficients
of an affine map preserving a numerical semigroup.

\begin{proposition}
Let $S$ be a numerical semigroup and let $a,b\in\mathbb{Z}$. If $S$ is closed under $T_{a,b}$, then $a\ge 0$ and $b\ge -(a-1)\operatorname{m}(S)$.
\end{proposition}

\begin{proof}
Let $m=\operatorname{m}(S)$. Since $S$ is a numerical semigroup, there
exists $N$ such that every integer $s\ge N$ belongs to $S$. If
$a\le -1$, choose $s\in S$ with $s>b/(-a)$. Then
$T_{a,b}(s)=as+b<0$, contradicting
$T_{a,b}(S\setminus\{0\})\subseteq S\setminus\{0\}\subseteq\mathbb{N}$.
Hence $a\ge 0$.

Now suppose that $b<-(a-1)m$. Then $T_{a,b}(m)=am+b<m$. Since $m$ is
the smallest positive element of $S$, this implies
$T_{a,b}(m)\notin S\setminus\{0\}$, again a contradiction. Therefore
$b\ge -(a-1)m$.
\end{proof}

Let $a,b,m\in\mathbb{Z}$ with $m>0$. We denote by $T^0_{a,b}$ the
identity map and by $T^i_{a,b}$ the $i$-fold composition of $T_{a,b}$.
The orbit of $m$ under $T_{a,b}$ is the sequence
$\Gamma=(T^i_{a,b}(m))_{i\ge 0}$. A direct computation gives
\[
T^i_{a,b}(m)=
\begin{cases}
a^im+\dfrac{a^i-1}{a-1}b, & \text{if } a\ne 1,\\[8pt]
m+ib, & \text{if } a=1.
\end{cases}
\]

For $a\ne 1$, set $d=(a-1)m+b$. Then the preceding formula can be
written as
\begin{equation}\label{eq:orbit_formula_d}
T^i_{a,b}(m)=m+\frac{a^i-1}{a-1}d,\qquad i\ge 0.
\end{equation}
Thus, when $a\ge 1$, the condition $b\ge -(a-1)m$ is equivalent to
$d\ge 0$.

\begin{lemma}\label{lem:orbit_increasing}
Let $a,b,m\in\mathbb{Z}$ with $a\ge 1$, $m>0$, and
$b\ge -(a-1)m$. Then $0<m\le T^i_{a,b}(m)\le T^{i+1}_{a,b}(m)$ for
every $i\ge 0$. Moreover, if $b>-(a-1)m$, then the sequence
$(T^i_{a,b}(m))_{i\ge 0}$ is strictly increasing.
\end{lemma}

\begin{proof}
If $a=1$, then $b\ge 0$ and $T^i_{1,b}(m)=m+ib$, so the assertion is
immediate. Assume now that $a>1$. By \eqref{eq:orbit_formula_d},
$T^i_{a,b}(m)=m+\frac{a^i-1}{a-1}d$. Since $d\ge 0$, the sequence is
nondecreasing and all its terms are at least $m>0$. If
$b>-(a-1)m$, then $d>0$, and the inequalities between consecutive terms
are strict.
\end{proof}

We now introduce the semigroups generated by such orbits. Let
$S_{a,b}(m)=\langle T^i_{a,b}(m)\mid i\ge 0\rangle$. Under the
hypotheses $a\ge 1$, $m>1$, $b\ge -(a-1)m$, and $\gcd(b,m)=1$, one has
$b\ne -(a-1)m$; indeed, equality would imply $m\mid b$, contradicting
$\gcd(b,m)=1$. Hence $b>-(a-1)m$, and Lemma~\ref{lem:orbit_increasing}
shows that the orbit of $m$ is a strictly increasing sequence of
positive integers.

Moreover,
$\gcd(T^i_{a,b}(m)\mid i\ge 0)=\gcd(m,b)$. Indeed, each
$T^i_{a,b}(m)$ belongs to the subgroup of $\mathbb{Z}$ generated by
$m$ and $b$, while $m=T^0_{a,b}(m)$ and
$b=T_{a,b}(m)-am$. Therefore $S_{a,b}(m)$ is a numerical semigroup if
and only if $\gcd(b,m)=1$.

The cases $a=0$ and $a=1$ are classical. For $a=0$, the construction
gives numerical semigroups of embedding dimension two, whereas for
$a=1$ it gives numerical semigroups of maximal embedding dimension,
provided that $\gcd(b,m)=1$. These cases have been extensively studied
(see, for instance, \cite{libro}). Also, $S_{a,b}(1)=\mathbb{N}$.
Consequently, in the rest of this section we focus on the case
$a\ge 2$ and $m>1$.

The following theorem extends \cite[Theorem~3.1]{Ugo17}, where the case
$b>0$ was considered, to all admissible values of $b$.

\begin{theorem}\label{th:1}
Let $a,b\in\mathbb{Z}$ and let $m>1$. Assume that $a\ge 2$,
$b\ge -(a-1)m$, and $\gcd(b,m)=1$. Then $S_{a,b}(m)$ is the smallest
numerical semigroup of multiplicity $m$ closed under $T_{a,b}$.
\end{theorem}

\begin{proof}
By the previous discussion, $S_{a,b}(m)$ is a numerical semigroup.
Since $b\ge -(a-1)m$, Lemma~\ref{lem:orbit_increasing} shows that the
smallest positive generator in the orbit is $m$. Thus
$\operatorname{m}(S_{a,b}(m))=m$.

We next prove that $S_{a,b}(m)$ is closed under $T_{a,b}$. Let
$s\in S_{a,b}(m)\setminus\{0\}$. Then
$s=\sum_{i=0}^n u_iT^i_{a,b}(m)$ for some $n\ge 0$ and some
$u_0,\ldots,u_n\in\mathbb{N}$, not all zero.

If $u_j>0$ for some $j \geq 0$, then, using
$b=T^{j+1}_{a,b}(m)-aT^j_{a,b}(m)$, we get
\[
\begin{aligned}
T_{a,b}(s)
&=\sum_{i=0}^n au_iT^i_{a,b}(m)+b\\
&=\sum_{i\ne j} au_iT^i_{a,b}(m)
  +a(u_j-1)T^j_{a,b}(m)+T^{j+1}_{a,b}(m).
\end{aligned}
\]
Hence $T_{a,b}(s)\in S_{a,b}(m)\setminus\{0\}$.

Finally, let $S$ be any numerical semigroup of multiplicity $m$ closed
under $T_{a,b}$. Since $m\in S$, closure under $T_{a,b}$ gives
$T^i_{a,b}(m)\in S$ for every $i\ge 0$. Therefore all generators of
$S_{a,b}(m)$ belong to $S$, and so $S_{a,b}(m)\subseteq S$.
\end{proof}

For $b>0$, Theorem~\ref{th:1} recovers the affine-closed numerical
semigroups studied by Ugolini \cite{Ugo17}. These semigroups also appear
in \cite{XIN} (see Collection \texttt{CNS}) after a suitable adjustment
of notation. The construction also corresponds to the one-map case of
the $(a,b)$-monoids introduced by Robles P\'erez and Rosales
\cite{RoblesRosales18}. The point of view adopted here naturally extends
the construction to the admissible nonpositive values of $b$, namely
those satisfying $b\ge -(a-1)m$.

Finally, under the hypotheses of Theorem~\ref{th:1}, set
$d=(a-1)m+b$. Since $m>1$ and $\gcd(b,m)=1$, one has $d>0$, and
\eqref{eq:orbit_formula_d} shows that
$T^i_{a,b}(m)=m+\frac{a^i-1}{a-1}d$ for every $i\ge 0$. Thus the
smallest numerical semigroup of multiplicity $m$ closed under
$T_{a,b}$ belongs to the shifted family in the sense of
\cite{ONeillPelayo18}.

\section{\texorpdfstring{The embedding dimension of $S_{a,b}(m)$}{The embedding dimension}}\label{Sect:e}

Throughout this section, let $a,b,m\in\mathbb{Z}$ with $a\ge 2$,
$m>1$, $\gcd(b,m)=1$, and $b>-(a-1)m$. We write
$S=S_{a,b}(m)$ and $T=T_{a,b}$. By Theorem~\ref{th:1}, $S$ is generated
by the orbit of $m$ under $T$. Thus, if $s_i=T^i(m)$ for $i\ge 0$, then
$S=\langle s_i\mid i\ge 0\rangle$.

Set $d=(a-1)m+b$ and $A_i=(a^i-1)/(a-1)$ for $i\ge 0$. Then $d>0$,
$\gcd(d,m)=1$, and, by \eqref{eq:orbit_formula_d}, one has
$s_i=m+A_id$ for every $i\ge 0$. In particular, the sequence
$(s_i)_{i\ge 0}$ is strictly increasing.

Our first goal is to give a general lower bound for the embedding
dimension of $S$.

\begin{proposition}\label{prop:3}
Let $n$ be the smallest positive integer such that $A_n\ge m$. Then
\[\{s_0,\ldots,s_{n-1}\}\subseteq \operatorname{msg}(S).\] In particular, $\operatorname{e}(S)\ge n$.
\end{proposition}

\begin{proof}
Since the sequence $(s_i)_{i\ge 0}$ is strictly increasing, no generator
$s_j$ with $j>i+1$ can occur in a factorization of $s_{i+1}$. Hence, by
\cite[Corollary~2.9]{libro}, it is enough to prove that
$s_{i+1}\notin\langle s_0,\ldots,s_i\rangle$ for
$i=0,\ldots,n-2$.

Suppose, to the contrary, that
$s_{i+1}\in\langle s_0,\ldots,s_i\rangle$ for some
$i\in\{0,\ldots,n-2\}$. Then there exist
$u_0,\ldots,u_i\in\mathbb{N}$ such that
\[
m+A_{i+1}d
=
s_{i+1}
=
\sum_{j=0}^i u_js_j
=
\left(\sum_{j=0}^i u_j\right)m
+
\left(\sum_{j=0}^i u_jA_j\right)d.
\]
Reducing modulo $d$ and using $\gcd(m,d)=1$, we obtain
$\sum_{j=0}^i u_j\equiv 1\pmod d$. Hence
$\sum_{j=0}^i u_j=1+dk$ for some $k\in\mathbb{N}$.

We claim that $k\ge 1$. Indeed, if $k=0$, then
$\sum_{j=0}^i u_j=1$, and the equality
$s_{i+1}=\sum_{j=0}^i u_js_j$ would force $s_{i+1}=s_j$ for some
$j\le i$, contradicting the fact that $(s_\ell)_{\ell\ge 0}$ is strictly
increasing.

Moreover, at least one coefficient $u_j$ with $j\ge 1$ is nonzero. If
not, then $s_{i+1}=u_0s_0=u_0m$, and so $m\mid A_{i+1}d$. Since
$\gcd(m,d)=1$, this gives $m\mid A_{i+1}$, which is impossible because
$0<A_{i+1}\le A_{n-1}<m$.

Thus $\sum_{j=0}^i u_j\ge 1+d$, and some $u_j$ with $j\ge 1$ is
positive. Therefore
\[
s_{i+1}
=
\left(\sum_{j=0}^i u_j\right)m
+
\left(\sum_{j=0}^i u_jA_j\right)d
>
\left(\sum_{j=0}^i u_j\right)m
\ge (1+d)m.
\]
Since $A_{n-1}<m$, we have $(1+d)m>m+dA_{n-1}=s_{n-1}$. Hence
$s_{i+1}>s_{n-1}$, contradicting $i+1\le n-1$ and the monotonicity of
the sequence $(s_\ell)_{\ell\ge 0}$. This proves that
$s_{i+1}\notin\langle s_0,\ldots,s_i\rangle$ for all
$i=0,\ldots,n-2$, and the result follows.
\end{proof}

The lower bound in Proposition~\ref{prop:3} is not always sharp. For
instance, if $a=3$, $b=-83$, and $m=42$, then $d=1$ and $S$ is minimally
generated by $\{42,43,46,55,82,163\}$. Hence $\operatorname{e}(S)=6$,
whereas $A_4=40<42<A_5=121$, so Proposition~\ref{prop:3} gives only
$\operatorname{e}(S)\ge 5$.

Nevertheless, the bound is sharp under a mild additional condition.

\begin{corollary}\label{cor:e=n}
Let $n$ be the smallest positive integer such that $A_n\ge m$, and write
$A_n=qm+r$, where $q\in\mathbb{N}\setminus\{0\}$ and
$0\le r\le m-1$. If
\[
b\ge -(a-1)m+\left\lceil\frac{r-1}{q}\right\rceil,
\]
then $\operatorname{msg}(S)=\{s_0,\ldots,s_{n-1}\}$ and
$\operatorname{e}(S)=n$.
\end{corollary}

\begin{proof}
Let $M=\langle s_0,\ldots,s_{n-1}\rangle$. Since $s_0=m$ and
$s_1=m+d$, we have $\gcd(s_0,s_1)=1$. Thus $M$ is a numerical semigroup.
Moreover, $M$ has multiplicity $m$, because $m\in M$ and all generators
$s_0,\ldots,s_{n-1}$ are at least $m$.

We first prove that $s_n\in M$. Since $s_0=m$, $s_1=m+d$, and
$A_n=qm+r$, we have
\[
s_n
=
m+A_nd
=
m+(qm+r)d
=
(dq+1-r)s_0+rs_1.
\]
By hypothesis, $d=(a-1)m+b\ge \lceil (r-1)/q\rceil$, and hence
$dq+1-r\ge 0$. Therefore $s_n\in M$.

We now show that $M$ is closed under $T$. Let
$s=\sum_{i=0}^{n-1}u_is_i\in M\setminus\{0\}$. If $s=u_0s_0$ with
$u_0>0$, then $T(s)=a(u_0-1)s_0+s_1\in M$. Otherwise, choose
$j\in\{1,\ldots,n-1\}$ with $u_j>0$. Since $b=s_{j+1}-as_j$, and since
$s_n\in M$ when $j=n-1$, we get
\[
T(s)
=
\sum_{i\ne j} au_is_i+a(u_j-1)s_j+s_{j+1}\in M.
\]
Thus $M$ is closed under $T$.

By Theorem~\ref{th:1}, $S$ is the smallest numerical semigroup of
multiplicity $m$ closed under $T$. Hence $S\subseteq M$. The reverse
inclusion is clear, because $M$ is generated by elements of $S$.
Therefore $S=M$.

Finally, Proposition~\ref{prop:3} gives
$\{s_0,\ldots,s_{n-1}\}\subseteq \operatorname{msg}(S)$, while
$S=\langle s_0,\ldots,s_{n-1}\rangle$. Hence
$\operatorname{msg}(S)=\{s_0,\ldots,s_{n-1}\}$ and
$\operatorname{e}(S)=n$.
\end{proof}

If $A_n\equiv 0\pmod m$ or $A_n\equiv 1\pmod m$, that is, if
$r=0$ or $r=1$, then $\lceil (r-1)/q\rceil\le 0$. Hence the hypothesis
of Corollary~\ref{cor:e=n} is automatically satisfied, and consequently
$\operatorname{e}(S)=n$ in these cases.

\begin{remark}\label{rem:e=n}
Let $n$ be the smallest positive integer such that $A_n\ge m$, and write
$A_n=qm+r$ with $q\ge 1$ and $0\le r\le m-1$. Since
$\lceil (r-1)/q\rceil\le r-1\le m-2$, the uniform condition
$d\ge m-2$ implies the hypothesis of Corollary~\ref{cor:e=n}. Therefore
$\operatorname{e}(S)=n$ whenever $d\ge m-2$. Equivalently, since
$d=(a-1)m+b$, this condition can be written as
$b\ge -(a-2)m-2$.
\end{remark}

\section{The Apéry set}\label{sec:apery}

Throughout this section, let $a,b,m\in\mathbb{Z}$ with $a\ge 2$,
$m>1$, $\gcd(b,m)=1$, and $b\ge -(a-2)m-2$. We write
$S=S_{a,b}(m)$.

As before, set $d=(a-1)m+b$ and, for $i\ge 0$, set
$A_i=(a^i-1)/(a-1)$. Thus $A_0=0$, and
$A_i=1+a+\cdots+a^{i-1}$ for $i\ge 1$. By
\eqref{eq:orbit_formula_d}, the generators arising from the orbit of
$m$ are $s_i=m+dA_i$, $i\ge 0$.

Let $n$ be the smallest positive integer such that $A_n\ge m$. Since
$d\ge m-2$, Remark~\ref{rem:e=n} applies. Therefore
$\operatorname{msg}(S)=\{s_0,\ldots,s_{n-1}\}$, and
$S=\langle s_0,\ldots,s_{n-1}\rangle$, where
$s_i=m+dA_i$ for $i=0,\ldots,n-1$. 

\subsection{Description of the Ap\'ery set}

We compute the Ap\'ery set of $S$ with respect to $m$, that is,
\[
\operatorname{Ap}(S,m):=\{s\in S \mid s-m\notin S\};
\] 
for simplicity, we write $\operatorname{Ap}(S)=\operatorname{Ap}(S,m)$. To this end, we first introduce an auxiliary set of exponent vectors.

\begin{definition}\label{def:conj_R}
Given an integer $i>1$, we define $R(a,i)$ to be the subset of $\mathbb{N}^{\,i-1}$ whose elements
$\mathbf{u}=(u_1,\ldots,u_{i-1})$ satisfy
\begin{itemize}
\item[(a)] $0\le u_j\le a$ for every $j=1,\ldots,i-1$;
\item[(b)] if $u_j=a$, then $u_k=0$ for every $k<j$.
\end{itemize}
\end{definition}

With this notation, define
\[
\mathcal{R}_m(a):=\left\{\mathbf{u}=(u_1,\ldots,u_{n-1})\in R(a,n)\ \middle|\ \sum_{i=1}^{n-1} u_i\,A_i < m\right\}.
\]

\begin{theorem}\label{Th:Ap}
With the above notation, we have
\[
\operatorname{Ap}(S)=\left\{\sum_{i=1}^{n-1} u_i s_i\ \middle|\ (u_1,\ldots,u_{n-1})\in \mathcal{R}_m(a)\right\}.
\]
\end{theorem}

To prove this theorem, we first need two preparatory results. The following result appears in \cite[Theorem~3.4]{Ugo17}; for the sake
of completeness, we include a proof.

\begin{proposition}\label{prop:canonical}
Let $n>1$ be an integer. For each $\ell\in[0,A_n)\cap\mathbb{N}$ there exists a unique
$\mathbf{u}^{(\ell)}=(u^{(\ell)}_1,\ldots,u^{(\ell)}_{n-1})\in R(a,n)$ such that
\[
\ell=\sum_{i=1}^{n-1}u^{(\ell)}_iA_i.
\]
\end{proposition}

\begin{proof}
We construct $\mathbf{u}^{(\ell)}$ by a greedy division algorithm.
Set $r_{n-1}:=\ell$. For $j=n-1,n-2,\ldots,1$ define
\[
u^{(\ell)}_j:=\left\lfloor \frac{r_j}{A_j}\right\rfloor,
\qquad
r_{j-1}:=r_j-u^{(\ell)}_jA_j.
\]
Since $r_j<A_{j+1}=aA_j+1$, we have $u^{(\ell)}_j\in\{0,1,\ldots,a\}$ for all $j$. If $u^{(\ell)}_j=a$, then $r_{j-1}=r_j-aA_j<1$, so $r_{j-1}=0$ and hence $u^{(\ell)}_{j-1}=\cdots=u^{(\ell)}_1=0$. Thus $\mathbf{u}^{(\ell)}\in R(a,n)$ and $\ell=\sum_{i=1}^{n-1}u^{(\ell)}_iA_i$ by construction.

Uniqueness follows from the same division step: $u^{(\ell)}_{n-1}=\lfloor \ell/A_{n-1}\rfloor$ is forced, and once it is fixed, the remainder $r_{n-2}=\ell-u^{(\ell)}_{n-1}A_{n-1}$ is uniquely determined. The argument repeats recursively.
\end{proof}

In what follows, $\mathbf{u}^{(\ell)}$ denotes the unique vector given by
Proposition~\ref{prop:canonical}. The next result records a minimality property
of this canonical representative.

Recall that the graded reverse lexicographical order on $\mathbb{N}^{n-1}$ is defined as follows: $\mathbf{u}=(u_1,\ldots,u_{n-1}) \preceq_{\mathrm{grevlex}} \mathbf{v}=(v_1,\ldots,v_{n-1})$ if and only if $\sum_{i=1}^{n-1}u_i<\sum_{i=1}^{n-1}v_i$, or $\sum_{i=1}^{n-1}u_i=\sum_{i=1}^{n-1}v_i$ and, letting $t=\max\{i\in\{1,\ldots,n-1\}\mid u_i\neq v_i\}$, one has $u_t>v_t$.

\begin{lemma}\label{lem:grevlex}
Let $\ell\in[0,A_n)\cap\mathbb{N}$. If
$\ell=\sum_{i=1}^{n-1}u_iA_i$ for some
$\mathbf{u}=(u_1,\ldots,u_{n-1})\in\mathbb{N}^{n-1}$, then
$\mathbf{u}^{(\ell)}\preceq_{\mathrm{grevlex}}\mathbf{u}$.
\end{lemma}

\begin{proof}
Let $\mathbf{u}=(u_1,\ldots,u_{n-1})\in\mathbb{N}^{n-1}$ be a representation of $\ell$, that is, $\ell=\sum_{i=1}^{n-1}u_iA_i$. We prove that, if $\mathbf{u}\notin R(a,n)$, then there exists another representation $\mathbf{u}'$ of $\ell$ such that $\mathbf{u}'\prec_{\mathrm{grevlex}}\mathbf{u}$.

Since $\mathbf{u}\notin R(a,n)$, choose the smallest index $j$ such that
either $u_j\ge a+1$, or $u_j=a$ and $u_k>0$ for some $k<j$. If
$u_j=a$, let $k<j$ be the smallest index with $u_k>0$; if
$u_j\ge a+1$, let $k\le j$ be the smallest index with $u_k>0$.

Notice that $j<n-1$. Indeed, if $j=n-1$, then either
$u_{n-1}\ge a+1$, or $u_{n-1}=a$ and some lower coordinate is positive.
In both cases $\sum_{i=1}^{n-1}u_iA_i\ge aA_{n-1}+1=A_n$,
contradicting $\ell<A_n$.
We use the identity $aA_j+A_k=A_{j+1}+aA_{k-1}$, with $A_0=0$. If $j\ne k$, define $\mathbf{u}'$ by setting
\[
u'_j=u_j-a,\quad u'_k=u_k-1,\quad u'_{j+1}=u_{j+1}+1,\quad u'_{k-1}=u_{k-1}+a,
\]
leaving all other coordinates unchanged and omitting the coordinate $u'_0$ when $k=1$. If $j=k$, define $\mathbf{u}'$ by setting
\[
u'_j=u_j-a-1,\quad u'_{j+1}=u_{j+1}+1,\quad u'_{j-1}=u_{j-1}+a,
\]
leaving all other coordinates unchanged and omitting the coordinate $u'_0$ when $j=1$. By the choice of $j$ and $k$, all coordinates of $\mathbf{u}'$ are nonnegative: indeed, if $j\ne k$, then $u_j\ge a$ and $u_k>0$, while if $j=k$, then $u_j\ge a+1$. Moreover, the above identity gives $\ell=\sum_{i=1}^{n-1}u'_iA_i$.

It remains to compare $\mathbf{u}'$ and $\mathbf{u}$. If $k=1$, then the total degree decreases, so $\mathbf{u}'\prec_{\mathrm{grevlex}}\mathbf{u}$. If $k>1$, then the total degree is preserved, and the right-most coordinate where $\mathbf{u}'$ and $\mathbf{u}$ differ is $j+1$, where $u'_{j+1}=u_{j+1}+1$; hence again $\mathbf{u}' \prec_{\mathrm{grevlex}} \mathbf{u}$.

Repeating this reduction process, we must eventually stop, because $\preceq_{\mathrm{grevlex}}$ is a well-order on $\mathbb{N}^{n-1}$. When the process stops, the resulting vector $\mathbf{v}$ belongs to $R(a,n)$ and still represents $\ell$. Hence, by Proposition~\ref{prop:canonical}, we have $\mathbf{v}=\mathbf{u}^{(\ell)}$. Since each reduction step decreases the grevlex order, it follows that $\mathbf{u}^{(\ell)} \preceq_{\mathrm{grevlex}} \mathbf{u}$.
\end{proof}

We can now prove the description of the Ap\'ery set.

\begin{proof}[Proof of Theorem~\ref{Th:Ap}]
By Proposition~\ref{prop:canonical}, the set $\mathcal{R}_m(a)$ has
cardinality $m$. Hence it suffices to prove that, for each residue class
modulo $m$, the corresponding element in the displayed set is the least
element of $S$ in that class.

Let $k\in\{0,\ldots,m-1\}$, and let $\ell_k$ be the unique integer with
$0\le \ell_k<m$ and $d\ell_k\equiv k\pmod m$. Write
$\ell_k=\sum_{i=1}^{n-1}u_i^{(\ell_k)}A_i$ as in
Proposition~\ref{prop:canonical}, and set
$W_k=\sum_{i=1}^{n-1}u_i^{(\ell_k)}s_i$. Then
$W_k\equiv d\ell_k\equiv k\pmod m$.

Let $w$ be the least element of $S$ congruent to $k$ modulo $m$. By
\cite[Lemma~2.4]{libro}, it is enough to prove that $w=W_k$. Since
$w\in\operatorname{Ap}(S,m)$, no factorization of $w$ uses $s_0=m$.
Thus
\[
w=\sum_{i=1}^{n-1}u_is_i
=
\left(\sum_{i=1}^{n-1}u_i\right)m
+d\left(\sum_{i=1}^{n-1}u_iA_i\right)
\]
for some $u_1,\ldots,u_{n-1}\in\mathbb{N}$. Since $w\equiv k\pmod m$,
there exists $N\in\mathbb{N}$ such that
$\sum_{i=1}^{n-1}u_iA_i=Nm+\ell_k$.

If $N=0$, Lemma~\ref{lem:grevlex} gives
$\mathbf{u}^{(\ell_k)}\preceq_{\mathrm{grevlex}}(u_1,\ldots,u_{n-1})$,
and hence $\sum_i u_i^{(\ell_k)}\le \sum_i u_i$. Therefore $W_k\le w$,
and the minimality of $w$ gives $W_k=w$.

Assume that $N\ge 1$. Then $\sum_i u_iA_i\ge m$. Since
$A_i\le A_{n-1}<m$ for all $i$, we have $\sum_i u_i\ge 2$. Also,
$\sum_i u_i^{(\ell_k)}\le \ell_k\le m-1$. As $d\ge m-2$, it follows
that
\[
dN+\sum_{i=1}^{n-1}u_i-1-\sum_{i=1}^{n-1}u_i^{(\ell_k)}\ge 0.
\]
Consequently,
\[
w-m
=
W_k+
\left(dN+\sum_{i=1}^{n-1}u_i-1-\sum_{i=1}^{n-1}u_i^{(\ell_k)}\right)m
\in S,
\]
contradicting $w\in\operatorname{Ap}(S,m)$. Thus $N=0$, and the previous
case yields $w=W_k$.

Therefore the displayed set contains the least representative of every
class modulo $m$, and hence it is $\operatorname{Ap}(S,m)$.
\end{proof}

We finally point out a useful comparison property of the canonical representatives.

\begin{corollary}\label{cor:order-u}
With the above notation, for any distinct $\ell,\ell'\in[0,A_n)\cap\mathbb{N}$ we have $\ell<\ell'$
if and only if the right-most nonzero entry of $\mathbf{u}^{(\ell')}-\mathbf{u}^{(\ell)}$ is positive.
\end{corollary}

\begin{proof}
Let $k=\max\{\,i\in\{1,\ldots,n-1\}\mid
u_i^{(\ell')}\neq u_i^{(\ell)}\,\}$. If $k=1$, then
$\ell'-\ell=(u^{(\ell')}_1-u^{(\ell)}_1)A_1$, so the claim is immediate.
Hence assume that $k\ge 2$.

Since
$(u_1^{(\ell)},\ldots,u_{k-1}^{(\ell)})$ and
$(u_1^{(\ell')},\ldots,u_{k-1}^{(\ell')})$ belong to $R(a,k)$, the
corresponding sums
$\sum_{i=1}^{k-1}u_i^{(\ell)}A_i$ and
$\sum_{i=1}^{k-1}u_i^{(\ell')}A_i$ lie in $[0,A_k)$. Hence
$D:=\sum_{i=1}^{k-1}(u_i^{(\ell')}-u_i^{(\ell)})A_i \in(-A_k,A_k)$.
Therefore
\[
\ell'-\ell=(u_k^{(\ell')}-u_k^{(\ell)})A_k+D.
\]
Since $u_k^{(\ell')}-u_k^{(\ell)}$ is a nonzero integer, the sign of
$\ell'-\ell$ is the sign of $u_k^{(\ell')}-u_k^{(\ell)}$, and the claim
follows.
\end{proof}

\subsection{Factorization lengths and classical invariants} 

We now derive some consequences of the Ap\'ery description for factorization lengths and classical invariants. First, recall that for a finite set $B\subset\mathbb{N}\setminus\{0\}$ the set of factorization lengths of
$s\in\mathbb{N}$ with respect to $B$ is
\[
\mathsf{L}_B(s):=\left\{\sum_{b\in B} u_b\ \middle|\ s=\sum_{b\in B} u_b\,b,\ u_b\in\mathbb{N}\right\},
\]
and $\mathsf{m}_B(s)$ denotes the minimum factorization length of $s$ with respect to $B$.

\begin{corollary}
For each $k\in\{0,\ldots,m-1\}$, let $w(k)$ be the element of $\operatorname{Ap}(S)$ congruent to $k$ modulo $m$. Then $w(k)$ factorizes as
\[
w(k)=\sum_{i=1}^{n-1}u^{(\ell_k)}_i s_i,
\]
where $\ell_k$ is the unique integer in $\{0,\ldots,m-1\}$ with $d\ell_k\equiv k\pmod m$ and $\mathbf{u}^{(\ell_k)}=(u^{(\ell_k)}_1,\ldots,u^{(\ell_k)}_{n-1})$ is the canonical representative of $\ell_k$. In particular,
\[
\mathsf{L}_{\operatorname{msg}(S)}(w(k))
=
\left\{\sum_{i=1}^{n-1}u^{(\ell_k)}_i\right\}.
\]
\end{corollary}

\begin{proof}
The displayed factorization is precisely the equality $w(k)=W_k$ proved in Theorem~\ref{Th:Ap}. Let $w(k)=\sum_{i=0}^{n-1}u_is_i$ be any factorization with respect to $\operatorname{msg}(S)$. Since $w(k)\in\operatorname{Ap}(S)$, we have $u_0=0$. Set $L=\sum_{i=1}^{n-1}u_i$ and $M=\sum_{i=1}^{n-1}u_iA_i$. Then $w(k)=Lm+dM$, and the congruence $w(k)\equiv k\pmod m$ gives $M=Nm+\ell_k$ for some $N\in\mathbb N$.

As in the proof of Theorem~\ref{Th:Ap}, the assumption $N\ge 1$ would imply $w(k)-m\in S$, which is impossible because $w(k)\in\operatorname{Ap}(S)$. Thus $N=0$. Hence $M=\ell_k$, and $w(k)=Lm+d\ell_k$. On the other hand, the canonical factorization gives $w(k)=\bigl(\sum_{i=1}^{n-1}u_i^{(\ell_k)}\bigr)m+d\ell_k$. Therefore $L=\sum_{i=1}^{n-1}u_i^{(\ell_k)}$, so every factorization of $w(k)$ has the same length.
\end{proof}

By the previous corollary, every $s\in\operatorname{Ap}(S)$ has a unique factorization length with respect to $\operatorname{msg}(S)$. Hence $S$ is homogeneous in the sense of \cite{JafariZarzuela18}.

Set $A:=\{A_i\mid i=1,\ldots,n-1\}$. By Proposition~\ref{prop:canonical} and Lemma~\ref{lem:grevlex}, the canonical representative of each $\ell\in[0,A_n)\cap\mathbb{N}$ has minimum length among all factorizations of $\ell$ with respect to $A$. Therefore
\begin{equation}\label{minFL}
\mathsf{m}_A(\ell)=\sum_{i=1}^{n-1}u_i^{(\ell)}.
\end{equation}

\begin{corollary}\label{cor.OP}
With the above notation,
\[
\operatorname{Ap}(S)=\left\{d\ell+\mathsf{m}_A(\ell)\,m\ \middle|\ 0\le \ell\le m-1\right\}.
\]
\end{corollary}

\begin{proof}
For each $k$, the element $w(k)$ satisfies $w(k)\equiv d\ell_k\pmod m$ and equals
\[
w(k)=d\ell_k+\left(\sum_{i=1}^{n-1}u_i^{(\ell_k)}\right)m
\]
by expanding $\sum_{i=1}^{n-1}u_i^{(\ell_k)}s_i$. Using~\eqref{minFL}, this becomes
$w(k)=d\ell_k+\mathsf{m}_A(\ell_k)m$. Since $\ell_k$ runs over $\{0,\ldots,m-1\}$ as $k$ does,
the result follows.
\end{proof}

The explicit description of $\operatorname{Ap}(S)$ in Theorem~\ref{Th:Ap} and Corollary~\ref{cor.OP} can be viewed as an analogue of \cite[Theorem~3.3]{ONeillPelayo18} for this shifted family, providing a closed description of Apéry sets without additional largeness assumptions.

\begin{corollary}\label{cor:FrobGen}
With the above notation, the Frobenius number and the genus of $S$ are given by
\begin{itemize}
    \item $\operatorname{F}(S) = \max_{0\le \ell\le m-1}\bigl\{d\ell+\mathsf{m}_A(\ell)\,m\bigr\}-m$;
    \item $g(S)=\sum_{\ell=0}^{m-1}\mathsf{m}_A(\ell)+\dfrac{(d-1)(m-1)}{2}$.
\end{itemize}
\end{corollary}

\begin{proof}
By Corollary~\ref{cor.OP}, $\operatorname{Ap}(S)=\{\,d\ell+\mathsf{m}_A(\ell)m \mid 0\le \ell\le m-1\,\}$.
Selmer's formulas for the Frobenius number and the genus in terms of the Ap\'ery set
(see, e.g., \cite[Proposition~2.12]{libro}) yield
$\operatorname{F}(S)=\max\operatorname{Ap}(S)-m$ and
\[
g(S)=\frac{1}{m}\sum_{\ell=0}^{m-1}\bigl(d\ell+\mathsf{m}_A(\ell)m\bigr)-\frac{m-1}{2}.
\]
The stated formula for the genus follows from
$\sum_{\ell=0}^{m-1}\ell=m(m-1)/2$.
\end{proof}

\begin{corollary}\label{cor:F_at_m-1}
With the above notation, if the maximum of $\{\,d\ell+\mathsf{m}_A(\ell)m \mid 0\le \ell\le m-1\,\}$ is attained at $\ell=m-1$, then
\begin{equation}\label{ecu.F.Ugo}
\operatorname{F}(S)=d(m-1)+\mathsf{m}_A(m-1)m-m.
\end{equation}
\end{corollary}

\begin{proof}
By Corollary~\ref{cor:FrobGen}, we have $\operatorname{F}(S)=\max_{0\le \ell\le m-1}\left(d\ell+\mathsf{m}_A(\ell)m\right)-m$. So, if the maximum is attained at $\ell=m-1$, then \eqref{ecu.F.Ugo} follows immediately.
\end{proof}

Corollary~\ref{cor:F_at_m-1} isolates the only input needed to obtain \eqref{ecu.F.Ugo}, namely that the maximum of
$\{d\ell+\mathsf{m}_A(\ell)m:0\le \ell\le m-1\}$ is attained at $\ell=m-1$.
In the case $b\ge 0$, this endpoint condition is verified in \cite[Theorem~3.4(3)]{Ugo17}; hence \eqref{ecu.F.Ugo} recovers the
corresponding Frobenius formula in \cite{Ugo17} (written in our notation). The following example illustrates that this condition is not
automatic in the feasible negative-$b$ range, and thus \eqref{ecu.F.Ugo} may fail.

\begin{example}
If $a=3$, $b=-43$ and $m=42$, then $b>-(a-2)m-2=-44$ and $S=S_{a,b}(m)$ is generated by $\{42,83,206,575,1682\}$.
The Frobenius number of $S$ is $1769$, whereas the right-hand side of \eqref{ecu.F.Ugo} is $1723$.
\end{example}

\section{Pseudo-Frobenius numbers}\label{sec:PF}

Let $S$ be a numerical semigroup. An element
$x\in\mathbb{Z}\setminus S$ is a pseudo-Frobenius number of $S$ if
$x+(S\setminus\{0\})\subseteq S$. We denote by $\operatorname{PF}(S)$
the set of pseudo-Frobenius numbers of $S$, and by $\operatorname{t}(S)$
its cardinality.

The following result is the case $n=\operatorname{m}(S)$ of
\cite[Proposition~2.20]{libro}.

\begin{proposition}\label{prop:PF1}
Let $S$ be a numerical semigroup and let $m=\operatorname{m}(S)$. Then
\[
\operatorname{PF}(S)=
\{\,w-m \mid w\in \operatorname{Maximals}_{\preceq_S}\operatorname{Ap}(S)\,\},
\]
where $x\preceq_S y$ if and only if $y-x\in S$.
\end{proposition}

Throughout this section, let $a,b,m\in\mathbb{Z}$ with $a\ge 2$,
$m>1$, $\gcd(b,m)=1$, and $b\ge -(a-2)m-2$. We keep the notation of
Section~\ref{sec:apery}: $S=S_{a,b}(m)$,
$d=(a-1)m+b$, $A_i=(a^i-1)/(a-1)$, and $s_i=m+dA_i$. Thus
$d\ge m-2$ and
$\operatorname{msg}(S)=\{s_0,\ldots,s_{n-1}\}$, where $n$ is the
smallest positive integer such that $A_n\ge m$. We also write
$A=\{A_i\mid 1\le i\le n-1\}$ and denote by $\mathsf{m}_A(\ell)$ the
minimum factorization length of $\ell$ with respect to $A$.

\begin{proposition}\label{prop:PF2}
With the above notation,
\[
\operatorname{PF}(S) \subseteq
\left\{\sum_{i=1}^{n-1} u_i s_i - m\ \middle|\ 
\mathbf{u}=(u_1,\ldots,u_{n-1}) \in
\operatorname{Maximals}_{\preceq_{\mathbb{N}^{n-1}}}\mathcal{R}_m(a)
\right\},
\]
where $\preceq_{\mathbb{N}^{n-1}}$ denotes the product order.
\end{proposition}

\begin{proof}
Let $x\in\operatorname{PF}(S)$. By Proposition~\ref{prop:PF1}, there
exists $w\in \operatorname{Maximals}_{\preceq_S}\operatorname{Ap}(S)$
such that $x=w-m$. By Theorem~\ref{Th:Ap}, write
$w=\sum_{i=1}^{n-1}u_i s_i$ with
$\mathbf{u}=(u_1,\ldots,u_{n-1})\in\mathcal{R}_m(a)$.

If $\mathbf{u}$ were not maximal in $\mathcal{R}_m(a)$ for the product
order, then there would exist
$\mathbf{v}=(v_1,\ldots,v_{n-1})\in\mathcal{R}_m(a)$ with
$\mathbf{u}\preceq_{\mathbb{N}^{n-1}}\mathbf{v}$ and
$\mathbf{u}\ne\mathbf{v}$. Again by
Theorem~\ref{Th:Ap}, $w'=\sum_{i=1}^{n-1}v_i s_i$ belongs to
$\operatorname{Ap}(S)$, and
$w'-w=\sum_{i=1}^{n-1}(v_i-u_i)s_i\in S\setminus\{0\}$, contradicting
the maximality of $w$. Thus $\mathbf{u}$ is maximal, and the result
follows.
\end{proof}

The reverse inclusion in Proposition~\ref{prop:PF2} does not hold in general.

\begin{example}
Let $a=3$, $b=-43$ and $m=44$. Then $S=S_{a,b}(m)$ is minimally generated by
$\{44,89,224,629,1844\}$, and
$\operatorname{PF}(S)=\{1886,1929,2067\}$. Moreover,
\[
\operatorname{Maximals}_{\preceq_{\mathbb{N}^{n-1}}}\mathcal{R}_m(a) = 
\{(0,0,3,0), (0,3,2,0), (3,2,2,0), (3,0,0,1)\}.
\]
The vector $(0,0,3,0)$ gives $3s_3-m=1843\notin \operatorname{PF}(S)$.
\end{example}

\subsection{A criterion for maximal Apéry elements}

Although the inclusion in Proposition~\ref{prop:PF2} is not an equality in general, it provides a finite set of candidates. We refine this approach using the parametrization of $\operatorname{Ap}(S)$ given in Corollary~\ref{cor.OP}.

For $0\le \ell\le m-1$, set
\[
W(\ell):=d\ell+\mathsf{m}_A(\ell)m.
\]
Thus $\operatorname{Ap}(S)=\{W(0),W(1),\ldots,W(m-1)\}$.

\begin{proposition}\label{prop:WplusSj_inApery}
Fix $\ell\in\{0,\ldots,m-1\}$ and $j\in\{1,\ldots,n-1\}$. Then $W(\ell)+s_j\in\operatorname{Ap}(S)$ if and only if $\ell+A_j<m$ and $\mathsf{m}_A(\ell+A_j)=\mathsf{m}_A(\ell)+1$.
\end{proposition}

\begin{proof}
Since $s_j=m+dA_j$ and $W(\ell)=d\ell+\mathsf{m}_A(\ell)m$, we have 
\begin{equation}\label{eq:WplusSj_short}
W(\ell)+s_j=d(\ell+A_j)+\left(\mathsf{m}_A(\ell)+1\right)m.
\end{equation}
The sufficiency is immediate from this identity.

Conversely, suppose that $W(\ell)+s_j\in\operatorname{Ap}(S)$, and let $r\in\{0,\ldots,m-1\}$ be such that $r\equiv \ell+A_j\pmod m$. Since $W(\ell)+s_j\equiv dr\pmod m$, the uniqueness of representatives in the Ap\'ery set gives $W(\ell)+s_j=W(r)$.

We claim that $r=\ell+A_j$. Indeed, since $A_j\le A_{n-1}<m$, if $r\ne\ell+A_j$, then $r=\ell+A_j-m\le A_j-1<m-1$. In this case, \eqref{eq:WplusSj_short} and $W(\ell)+s_j=W(r)$ give
\[
\mathsf{m}_A(r)=\mathsf{m}_A(\ell)+1+d\ge d+1\ge m-1>r,
\]
contradicting the bound $\mathsf m_A(r)\le r$, which follows from $A_1=1$. Hence $r=\ell+A_j$, and comparison in \eqref{eq:WplusSj_short} gives $\mathsf{m}_A(\ell+A_j)=\mathsf{m}_A(\ell)+1$.
\end{proof}

\begin{theorem}\label{th:maxcrit_mA}
For $\ell\in\{0,\ldots,m-1\}$, we have $W(\ell) \in \operatorname{Maximals}_{\preceq_S}(\operatorname{Ap}(S))$ if and only if $\mathsf{m}_A(\ell+A_j)\le \mathsf{m}_A(\ell)$ for every $j\in\{1,\ldots,n-1\}$ such that $\ell+A_j<m$.
\end{theorem}

\begin{proof}
Let $w\in\operatorname{Ap}(S)$. Since $\operatorname{msg}(S)=\{m,s_1,\ldots,s_{n-1}\}$, we have $w \in \operatorname{Maximals}_{\preceq_S}(\operatorname{Ap}(S))$ if and only if $w+s_j\notin\operatorname{Ap}(S)$ for all $j=1,\ldots,n-1$. Indeed, if $w+y\in\operatorname{Ap}(S)$ for some $y\in S\setminus\{0\}$, then $y$ cannot be a positive multiple of $m$, since otherwise $w+y-m\in S$. Hence some factorization of $y$ involves a generator $s_j$ with $j\in\{1,\ldots,n-1\}$. For such a $j$, we have $w+s_j\in\operatorname{Ap}(S)$; otherwise $w+s_j-m\in S$, and then $w+y-m\in S$, a contradiction.

Now take $w=W(\ell)$. By Proposition~\ref{prop:WplusSj_inApery}, $W(\ell)+s_j\in\operatorname{Ap}(S)$ if and only if $\ell+A_j<m$ and $\mathsf{m}_A(\ell+A_j)=\mathsf{m}_A(\ell)+1$. For $\ell+A_j<m$, one always has $\mathsf{m}_A(\ell+A_j)\le \mathsf{m}_A(\ell)+1$. Hence the latter equality fails exactly when $\mathsf{m}_A(\ell+A_j)\le \mathsf{m}_A(\ell)$, and the result follows.
\end{proof}

\begin{corollary}\label{cor:mminus1_maximal}
We have $W(m-1) \in \operatorname{Maximals}_{\preceq_S}(\operatorname{Ap}(S))$. In particular, $W(m-1)-m\in\operatorname{PF}(S)$.
\end{corollary}

\begin{proof}
For every $j\ge 1$ we have $(m-1)+A_j\ge m$, so the condition in Theorem~\ref{th:maxcrit_mA} is vacuous. The pseudo-Frobenius claim follows from Proposition~\ref{prop:PF1}.
\end{proof}

\subsection{Localizing maximal Apéry elements}

We now use Theorem~\ref{th:maxcrit_mA} to determine where the maximal elements of $\operatorname{Ap}(S)$ can appear.

\begin{proposition}\label{prop:tail_localization}
If $W(\ell) \in \operatorname{Maximals}_{\preceq_S}(\operatorname{Ap}(S))$, then $\ell+A_{n-1}\ge m$.
\end{proposition}

\begin{proof}
Assume that $\ell+A_{n-1}<m$. We show that
$\mathsf{m}_A(\ell+A_{n-1})=\mathsf{m}_A(\ell)+1$, and then apply
Proposition~\ref{prop:WplusSj_inApery} with $j=n-1$.

Let $\mathbf u^{(\ell)}=(u_1^{(\ell)},\ldots,u_{n-1}^{(\ell)})$. Since
$\ell+A_{n-1}<m\le A_n=aA_{n-1}+1$, the vector
$\mathbf u^{(\ell)}+\mathbf e_{n-1}$ belongs to $R(a,n)$ and represents
$\ell+A_{n-1}$. Indeed, if $u_{n-1}^{(\ell)}=a$, then
$\ell+A_{n-1}\ge (a+1)A_{n-1}\ge A_n\ge m$, a contradiction. The only
remaining possible obstruction would be $u_{n-1}^{(\ell)}=a-1$ together
with some positive lower coordinate; but then
$\ell+A_{n-1}\ge aA_{n-1}+1=A_n\ge m$, again a contradiction.

Thus, by Proposition~\ref{prop:canonical},
$\mathbf u^{(\ell+A_{n-1})}=\mathbf u^{(\ell)}+\mathbf e_{n-1}$. Hence,
using \eqref{minFL},
$\mathsf{m}_A(\ell+A_{n-1})=\mathsf{m}_A(\ell)+1$. Therefore
$W(\ell)+s_{n-1}\in\operatorname{Ap}(S)$ by
Proposition~\ref{prop:WplusSj_inApery}, contradicting the maximality of
$W(\ell)$.
\end{proof}

Proposition~\ref{prop:tail_localization} shows that maximal elements can only occur near the upper end of $\{0,\ldots,m-1\}$. In terms of canonical representatives, only the two largest possible values of the last coordinate can occur.

We can express this localization in terms of the last coordinate of the
canonical representatives. For $\ell\in\{0,\ldots,m-1\}$, write
$\ell=u_{n-1}^{(\ell)}A_{n-1}+r(\ell)$, where
\[r(\ell):=\sum_{i=1}^{n-2}u_i^{(\ell)}A_i.\] In particular,
$m-1=u_{n-1}^{(m-1)}A_{n-1}+r(m-1)$.

\begin{corollary}\label{cor:two_layers}
If $W(\ell) \in \operatorname{Maximals}_{\preceq_S}(\operatorname{Ap}(S))$, then
$u_{n-1}^{(\ell)}\in\{u_{n-1}^{(m-1)}-1,u_{n-1}^{(m-1)}\}$.
Equivalently, if $u_{n-1}^{(\ell)}\le u_{n-1}^{(m-1)}-2$, then
$W(\ell)$ is not maximal.
\end{corollary}

\begin{proof}
Since $A_{n-1}<m$, we have $u_{n-1}^{(m-1)}\ge 1$. By the greedy
construction in Proposition~\ref{prop:canonical},
$u_{n-1}^{(\ell)}=\lfloor \ell/A_{n-1}\rfloor$ for
$0\le \ell\le m-1$. If $W(\ell)$ is maximal, then
Proposition~\ref{prop:tail_localization} gives
$\ell\ge m-A_{n-1}=(u_{n-1}^{(m-1)}-1)A_{n-1}+r(m-1)+1$.
Since also
$\ell\le m-1=u_{n-1}^{(m-1)}A_{n-1}+r(m-1)$, it follows that
$\lfloor \ell/A_{n-1}\rfloor\in
\{u_{n-1}^{(m-1)}-1,u_{n-1}^{(m-1)}\}$. Hence the result follows.
\end{proof}

We now separate the two possible values of the last coordinate.

\begin{corollary}\label{cor:two_bands}
For $\ell\in\{0,\ldots,m-1\}$, if $W(\ell) \in \operatorname{Maximals}_{\preceq_S}(\operatorname{Ap}(S))$, then exactly one of the following
conditions holds:
\begin{enumerate}
\item[\rm(I)] $u_{n-1}^{(\ell)}=u_{n-1}^{(m-1)}$ and
$r(\ell)\le r(m-1)$;
\item[\rm(II)] $u_{n-1}^{(\ell)}=u_{n-1}^{(m-1)}-1$ and
$r(\ell)\ge r(m-1)+1$.
\end{enumerate}
\end{corollary}

\begin{proof}
Let $W(\ell)$ be maximal. By Corollary~\ref{cor:two_layers},
$u_{n-1}^{(\ell)}\in
\{u_{n-1}^{(m-1)}-1,u_{n-1}^{(m-1)}\}$. If
$u_{n-1}^{(\ell)}=u_{n-1}^{(m-1)}$, then
$\ell=u_{n-1}^{(m-1)}A_{n-1}+r(\ell)$ and
$m-1=u_{n-1}^{(m-1)}A_{n-1}+r(m-1)$. Since $\ell\le m-1$, we get
$r(\ell)\le r(m-1)$. Otherwise,
$u_{n-1}^{(\ell)}=u_{n-1}^{(m-1)}-1$. By
Proposition~\ref{prop:tail_localization}, $\ell+A_{n-1}\ge m$, and
therefore
$u_{n-1}^{(m-1)}A_{n-1}+r(\ell)\ge
u_{n-1}^{(m-1)}A_{n-1}+r(m-1)+1$. Thus
$r(\ell)\ge r(m-1)+1$. The two alternatives are mutually exclusive, so
exactly one holds.
\end{proof}

\section{The symmetric family}\label{sec:symmetric_family}

Throughout this section we keep the standing hypotheses
$a\ge 2$, $m>1$, $\gcd(b,m)=1$, and $b\ge -(a-2)m-2$, together with the
notation introduced in Sections~\ref{sec:apery} and~\ref{sec:PF}. Thus
$S=S_{a,b}(m)$, $d=(a-1)m+b$, and $n$ is the smallest positive integer
such that $A_n\ge m$.

We characterize when $S$ is symmetric. Recall that this is equivalent to
$\operatorname{t}(S)=1$, or equivalently to
$\operatorname{PF}(S)=\{\operatorname{F}(S)\}$
\cite[Corollary~4.11]{libro}.

\begin{theorem}\label{th:type-one}
The semigroup $S_{a,b}(m)$ is symmetric if and only if either $n=2$, or $n\ge 3$, $u_1^{(m-1)}=a$, and $u_i^{(m-1)}=a-1$ for every $i=2,\ldots,n-2$.
\end{theorem}

The proof relies on the construction of a distinguished maximal element of the Ap\'ery set and on a comparison with the endpoint $W(m-1)$.

\begin{proposition}\label{prop:vector_a_a-1}
Assume that $n\ge 3$. Then there exists $\ell\in\{0,\ldots,m-1\}$ such that $W(\ell) \in \operatorname{Maximals}_{\preceq_S}(\operatorname{Ap}(S))$, $u_1^{(\ell)}=a$, $u_i^{(\ell)}=a-1$ for $2\le i\le n-2$, and $u_{n-1}^{(\ell)}\in \{u_{n-1}^{(m-1)}-1,u_{n-1}^{(m-1)}\}$.
\end{proposition}

\begin{proof}
Set $c:=u_{n-1}^{(m-1)}$ and $r_0:=aA_1+(a-1)\sum_{i=2}^{n-2}A_i$, where the sum is empty if
$n=3$. Then $0<r_0<A_{n-1}$. Choose 
\[
q=
\begin{cases}
c, & \text{if } r_0\le r(m-1),\\
c-1, & \text{if } r_0\ge r(m-1)+1,
\end{cases}
\]
and set $\ell=qA_{n-1}+r_0$.

We first check that $\ell\in\{0,\ldots,m-1\}$. Since $A_{n-1}<m$, we
have $c\ge 1$, and so $q\ge 0$. If $q=c$, then
$\ell=cA_{n-1}+r_0\le cA_{n-1}+r(m-1)=m-1$. If $q=c-1$, then
$\ell=(c-1)A_{n-1}+r_0\le cA_{n-1}-1\le m-1$.

We claim that $(a,a-1,\ldots,a-1,q)$ is the canonical representative of
$\ell$. It clearly represents $\ell$. Moreover, it belongs to $R(a,n)$.
Indeed, if $q=c$, then $c<a$: otherwise, since
$\mathbf u^{(m-1)}\in R(a,n)$, the equality $c=a$ would force
$r(m-1)=0$, contradicting $0<r_0\le r(m-1)$. If $q=c-1$, then $q<a$
is immediate. Thus Proposition~\ref{prop:canonical} proves the claim. In particular, the last coordinate is $q\in\{c-1,c\}$.

It remains to show that $W(\ell)$ is maximal. By
Theorem~\ref{th:maxcrit_mA}, it suffices to prove that
$\mathsf m_A(\ell+A_j)\le \mathsf m_A(\ell)$ for every
$j\in\{1,\ldots,n-1\}$ such that $\ell+A_j<m$.

The choice of $q$ gives $\ell+A_{n-1}\ge m$. If $q=c$, then
$\ell+A_{n-1}=cA_{n-1}+r_0+A_{n-1}\ge cA_{n-1}+r(m-1)+1=m$; if
$q=c-1$, then $r_0\ge r(m-1)+1$, and hence
$\ell+A_{n-1}=cA_{n-1}+r_0\ge cA_{n-1}+r(m-1)+1=m$. Therefore every
admissible $j$ satisfies $j\le n-2$.

Let $j\in\{1,\ldots,n-2\}$ with $\ell+A_j<m$. Adding $A_j$ to the
canonical representative of $\ell$ gives a representation of
$\ell+A_j$. If $j=1$, the identity $(a+1)A_1=A_2$ gives a representation
of length at most $\mathsf m_A(\ell)$. If $2\le j\le n-2$, then the
$j$-th coordinate becomes $a$ while the first coordinate is positive;
using $aA_j+A_1=A_{j+1}$, we again obtain a representation of length at
most $\mathsf m_A(\ell)$. Hence
$\mathsf m_A(\ell+A_j)\le \mathsf m_A(\ell)$ for every admissible $j$,
and Theorem~\ref{th:maxcrit_mA} implies that $W(\ell)$ is maximal.
\end{proof}

The following technical lemma will be used to compare minimum factorization lengths below a fixed integer $\ell$.

\begin{lemma}\label{lem:subtract_special}
Assume that $n\ge 3$, and let $\ell\in[0,A_n)\cap\mathbb{N}$ be such that $\mathbf u^{(\ell)}=(a,a-1,\ldots,a-1,q)$ for some $q<a$. Then, for every $i\in\{0,\ldots,\ell-1\}$, we have
\[
\mathsf{m}_A(\ell-i)\le \mathsf{m}_A(\ell)-\mathsf{m}_A(i).
\]
\end{lemma}

\begin{proof}
Set $\mathbf v:=\mathbf u^{(\ell)}-\mathbf u^{(i)}
=(v_1,\ldots,v_{n-1})\in\mathbb Z^{n-1}$. Since $i<\ell$,
Corollary~\ref{cor:order-u} implies that the right-most nonzero coordinate of
$\mathbf v$ is positive. Moreover,
$\ell-i=\sum_{t=1}^{n-1}v_tA_t$.

If $\mathbf v\in\mathbb N^{n-1}$, then $\mathbf v$ is a nonnegative
representation of $\ell-i$. Therefore, by the definition of $\mathsf m_A$ and by~\eqref{minFL},
\[
\mathsf m_A(\ell-i)\le \sum_{t=1}^{n-1}v_t
=
\sum_{t=1}^{n-1}u_t^{(\ell)}
-
\sum_{t=1}^{n-1}u_t^{(i)}
=
\mathsf m_A(\ell)-\mathsf m_A(i),
\]
and the desired inequality follows.

Assume now that $\mathbf v\notin\mathbb N^{n-1}$. Since
$\mathbf u^{(\ell)}=(a,a-1,\ldots,a-1,q)$ with $q<a$, the last coordinate
$v_{n-1}$ cannot be negative; otherwise, it would be the right-most nonzero
coordinate of $\mathbf v$, contradicting Corollary~\ref{cor:order-u}. Moreover,
$v_1\ge0$, since $u_1^{(i)}\le a=u_1^{(\ell)}$. Hence any negative coordinate
must occur among the positions $2,\ldots,n-2$; in particular, this case cannot
occur when $n=3$.

Thus $n\ge4$. Let $\alpha\in\{2,\ldots,n-2\}$ be a negative coordinate. Then
necessarily $u_\alpha^{(i)}=a$, and so $v_\alpha=-1$. Since
$\mathbf u^{(i)}\in R(a,n)$, we have $u_t^{(i)}=0$ for every $t<\alpha$.
Hence $v_1=a$ and $v_t=a-1$ for $2\le t\le\alpha-1$.

We claim that $\alpha$ is the unique negative coordinate of $\mathbf v$. Indeed,
if $v_t<0$ for some $\alpha<t\le n-2$, then $u_t^{(i)}=a$, because
$u_t^{(\ell)}=a-1$. By the defining condition of $R(a,n)$ applied at $t$, we
would have $u_\alpha^{(i)}=0$, a contradiction. Thus $\alpha$ is unique.

Since the right-most nonzero coordinate of $\mathbf v$ is positive and
$v_\alpha<0$, there exists $\beta>\alpha$ such that $v_\beta>0$. The identity
needed below follows from $A_j=\sum_{h=0}^{j-1}a^h$ and
$(a-1)A_j=a^j-1$:

\[
\begin{aligned}
aA_1+(a-1)\sum_{t=2}^{\alpha-1}A_t-A_\alpha+A_\beta
&=
a+\sum_{t=2}^{\alpha-1}(a^t-1)-\sum_{h=0}^{\alpha-1}a^h+A_\beta \\
&=
A_\beta-(\alpha-1) \\
&=
\sum_{h=1}^{\beta-1}a^h-(\alpha-2) \\
\end{aligned}\]
\[\begin{aligned}
 \hspace{5.6cm}   &=
\sum_{h=1}^{\beta-\alpha+1}a^h
+
\sum_{h=\beta-\alpha+2}^{\beta-1}(a^h-1) \\
&=
aA_{\beta-\alpha+1}
+
(a-1)\sum_{h=\beta-\alpha+2}^{\beta-1}A_h,
\end{aligned}
\]
where sums with lower index greater than upper index are omitted.

Using this identity, in the integer representation of $\ell-i$ given by $\mathbf v$ we replace the contribution
$aA_1+(a-1)\sum_{t=2}^{\alpha-1}A_t-A_\alpha+A_\beta$
by $aA_{\beta-\alpha+1} + (a-1)\sum_{h=\beta-\alpha+2}^{\beta-1}A_h$. This produces a vector $\mathbf w\in\mathbb N^{n-1}$ still representing
$\ell-i$: before the second contribution is added, the coordinates
$1,\ldots,\alpha$ become zero, the coordinate $\beta$ decreases by one and
remains nonnegative because $v_\beta>0$, and all other coordinates remain
nonnegative because $\alpha$ is the unique negative coordinate of $\mathbf v$;
then the added right-hand side contributes only nonnegative coefficients.

Moreover, the replacement preserves the total coefficient sum, since both sides
of the identity have coefficient sum $a+(\alpha-2)(a-1)$. Hence
$\sum_{t=1}^{n-1}w_t=\sum_{t=1}^{n-1}v_t$. Therefore, again by the definition of $\mathsf m_A$ and by~\eqref{minFL},
\[
\mathsf m_A(\ell-i)
\le
\sum_{t=1}^{n-1}w_t
=
\sum_{t=1}^{n-1}v_t
=
\mathsf m_A(\ell)-\mathsf m_A(i),
\]
as claimed.
\end{proof}

\begin{corollary}\label{cor:32}
Assume that $n\ge 3$, and let $\ell\in\{1,\ldots,m-1\}$ be the integer constructed in the proof of
Proposition~\ref{prop:vector_a_a-1}. Then
$W(i)\preceq_S W(\ell)$ for every $i\in\{0,\ldots,\ell-1\}$.
\end{corollary}

\begin{proof}
Let $i\in\{0,\ldots,\ell-1\}$. By Lemma~\ref{lem:subtract_special}, we have
$\mathsf{m}_A(\ell-i)\le \mathsf{m}_A(\ell)-\mathsf{m}_A(i)$. Hence, by the definition of $W$,
$W(\ell)-W(i)=W(\ell-i)+\bigl(\mathsf{m}_A(\ell)-\mathsf{m}_A(i)-\mathsf{m}_A(\ell-i)\bigr)m$.
Since $\ell-i\in\{1,\ldots,m-1\}$, Corollary~\ref{cor.OP} gives
$W(\ell-i)\in\operatorname{Ap}(S)\subseteq S$, and the coefficient of $m$ is nonnegative.
Therefore $W(\ell)-W(i)\in S$, which means $W(i)\preceq_S W(\ell)$.
\end{proof}

We are now ready to prove the characterization.

\begin{proof}[Proof of Theorem \ref{th:type-one}]
By the equivalence between symmetry and type one recalled above, it is enough
to prove the stated characterization for $\operatorname{t}(S)=1$.

If $n=2$, then $A=\{A_1\}=\{1\}$. Hence $W(i)=i(d+m)=is_1$ for every
$0\le i\le m-1$, and therefore $W(i)\preceq_S W(j)$ whenever
$0\le i\le j\le m-1$. Since $\operatorname{Ap}(S)=\{W(0),\ldots,W(m-1)\}$,
we get that $W(m-1)$ is the unique maximal element of $\operatorname{Ap}(S)$.
By Proposition~\ref{prop:PF1}, $\operatorname{t}(S)=1$.

Assume now that $n\ge 3$, and set $c:=u_{n-1}^{(m-1)}$. Suppose first
that $u_1^{(m-1)}=a$ and $u_i^{(m-1)}=a-1$ for every
$i=2,\ldots,n-2$. Then
$\mathbf u^{(m-1)}=(a,a-1,\ldots,a-1,c)$. Moreover, $1\le c\le a-1$:
indeed, $c=0$ would imply $m-1<A_{n-1}$, while $c=a$ is impossible
because $\mathbf u^{(m-1)}\in R(a,n)$ and its first coordinate is $a$. 

Let $r_0$ be the
integer used in the proof of Proposition~\ref{prop:vector_a_a-1}, namely
$r_0=aA_1$ if $n=3$, and
$r_0=aA_1+(a-1)\sum_{i=2}^{n-2}A_i$ if $n>3$. Since
$\mathbf{u}^{(m-1)}=(a,a-1,\ldots,a-1,c)$, we have $r(m-1)=r_0$. Hence the
construction in Proposition~\ref{prop:vector_a_a-1} gives $q=c$ and
$\ell=cA_{n-1}+r_0=cA_{n-1}+r(m-1)=m-1$. Therefore, by
Corollary~\ref{cor:32}, $W(i)\preceq_S W(m-1)$ for every $0\le i<m-1$.
Moreover, $W(i)\ne W(m-1)$ for every $i<m-1$, since
$W(r)\equiv dr\pmod m$ and $\gcd(d,m)=1$. Thus no $W(i)$ with $i<m-1$ is
maximal. Since $W(m-1)$ is maximal by Corollary~\ref{cor:mminus1_maximal},
and $\operatorname{Ap}(S)=\{W(0),\ldots,W(m-1)\}$, we conclude that
$W(m-1)$ is the unique maximal element of $\operatorname{Ap}(S)$. By
Proposition~\ref{prop:PF1}, $\operatorname{t}(S)=1$.

Conversely, assume that $\operatorname{t}(S)=1$. By
Corollary~\ref{cor:mminus1_maximal}, $W(m-1) \in \operatorname{Maximals}_{\preceq_S}(\operatorname{Ap}(S))$. Since $\operatorname{t}(S)=1$,
Proposition~\ref{prop:PF1} implies that $W(m-1)$ is the unique maximal
element of $\operatorname{Ap}(S)$.

By Proposition~\ref{prop:vector_a_a-1}, there exists
$\ell\in\{0,\ldots,m-1\}$ such that $W(\ell)$ is maximal and
$\mathbf{u}^{(\ell)}=(a,a-1,\ldots,a-1,q)$ for some $q\in\{c-1,c\}$. By uniqueness
of the maximal element, $W(\ell)=W(m-1)$. Since $W(r)\equiv dr\pmod m$ and
$\gcd(d,m)=1$, this gives $\ell\equiv m-1\pmod m$. As
$\ell,m-1\in\{0,\ldots,m-1\}$, we obtain $\ell=m-1$. Consequently, $\mathbf{u}^{(m-1)}=(a,a-1,\ldots,a-1,q)$. By definition
$c=u^{(m-1)}_{n-1}$, so $q=c$, and therefore
$\mathbf{u}^{(m-1)}=(a,a-1,\ldots,a-1,c)$.

Finally, since $\mathbf{u}^{(m-1)}\in R(a,n)$ and its first coordinate is $a$, the defining
condition of $R(a,n)$ gives $c\le a-1$. Moreover, $c\ge 1$, because
$m-1\ge A_{n-1}$. Hence $c\in\{1,\ldots,a-1\}$.
\end{proof}

Theorem~\ref{th:type-one} gives a complete characterization of symmetry
in this family. In particular, for $n\ge 4$, symmetry is equivalent to
the existence of $c\in\{1,\ldots,a-1\}$ such that
\[
\mathbf u^{(m-1)}=(a,a-1,\ldots,a-1,c).
\]
Equivalently, $m=(c+1)A_{n-1}-n+3$ for some
$c\in\{1,\ldots,a-1\}$. In this case, Corollary~\ref{cor:FrobGen} gives
$\operatorname{F}(S) = d(m-1)+\bigl((n-2)(a-1)+c\bigr)m$.

\section{Bound for the type}\label{sec:type_bound}

We keep the notation of Section~\ref{sec:PF}. The goal of this section is to prove the following bound, which shows that the type is controlled by the embedding dimension.

\begin{theorem}\label{thm:type-bound}
With the notation of Section~\ref{sec:PF}, the type of $S$ satisfies $\operatorname{t}(S)\le n-1$.
\end{theorem}

For $k\ge 1$, set $A^{(k)}:=\{A_1,\ldots,A_k\}$. Whenever canonical representatives with respect to $A_1,\ldots,A_k$ are used, Proposition~\ref{prop:canonical} and Lemma~\ref{lem:grevlex} are applied with $k+1$ in place of $n$. Thus, for each $0\le x<A_{k+1}$, the number $\mathsf m_{A^{(k)}}(x)$ is the sum of the coordinates of the canonical representative of $x$ with respect to $A_1,\ldots,A_k$. We shall use this fact without further mention.

We first record three elementary properties of the minimum length functions $\mathsf m_{A^{(k)}}$.

\begin{lemma}\label{lem:full-level-canonical-vectors}
Let $k\ge 2$. For each $q\in\{0,\ldots,k-1\}$, the canonical representative of $A_{k+1}-1-q$ with respect to $A_1,\ldots,A_k$ is
\[
\mathbf u^{(A_{k+1}-1-q)}=
\begin{cases}
(0,\ldots,0,a), & q=0,\\[1mm]
(\underbrace{0,\ldots,0}_{k-q-1\text{ times}},\,a,\,
\underbrace{a-1,\ldots,a-1}_{q-1\text{ times}},\,a-1), & 1\le q\le k-1.
\end{cases}
\]
In particular, $\mathsf m_{A^{(k)}}(A_{k+1}-1-q)=a+q(a-1)$ for all $q=0,\ldots,k-1$.
\end{lemma}

\begin{proof}
For $q=0$, we have $A_{k+1}-1=aA_k$, and the vector $(0,\ldots,0,a)$ represents $A_{k+1}-1$. Since it belongs to $R(a,k+1)$, it is the canonical representative.

Let now $1\le q\le k-1$, and consider the vector
\[
\mathbf v_q=(\underbrace{0,\ldots,0}_{k-q-1\text{ times}},\,a,\,
\underbrace{a-1,\ldots,a-1}_{q-1\text{ times}},\,a-1).
\]
This vector belongs to $R(a,k+1)$: its only coordinate equal to $a$ occurs at position $k-q$, and all previous coordinates are zero. Moreover,
\[
\sum_{i=1}^k(\mathbf v_q)_iA_i
=(a-1)A_k+aA_{k-q}+\sum_{i=k-q+1}^{k-1}(a-1)A_i.
\]
Using $(a-1)A_i=a^i-1$, the last two terms give $A_k-q$. Hence the above sum is $(a-1)A_k+A_k-q=aA_k-q=A_{k+1}-1-q$. Thus $\mathbf v_q$ is the canonical representative of $A_{k+1}-1-q$. The formula for $\mathsf m_{A^{(k)}}(A_{k+1}-1-q)$ follows by summing the coordinates.
\end{proof}

\begin{lemma}\label{lem:splitting-lambda}
Let $k\ge 2$. If $0\le q\le a$, $0\le r<A_k$, and $qA_k+r<A_{k+1}$, then
\[
\mathsf m_{A^{(k)}}(qA_k+r)=\mathsf m_{A^{(k-1)}}(r)+q.
\]
\end{lemma}

\begin{proof}
Let $\mathbf u=(u_1,\ldots,u_{k-1})\in R(a,k)$ be the canonical representative of $r$ with respect to $A_1,\ldots,A_{k-1}$. We claim that $(u_1,\ldots,u_{k-1},q)$ belongs to $R(a,k+1)$. This is clear if $q<a$. If $q=a$, then the inequality $qA_k+r<A_{k+1}=aA_k+1$ forces $r=0$, and hence $\mathbf u=(0,\ldots,0)$, so the claim also holds.

The vector $(u_1,\ldots,u_{k-1},q)$ represents $qA_k+r$. Hence, by the uniqueness of canonical representatives, it is the canonical representative of $qA_k+r$ with respect to $A_1,\ldots,A_k$. Therefore, $\mathsf m_{A^{(k)}}(qA_k+r)=\sum_{i=1}^{k-1}u_i+q=\mathsf m_{A^{(k-1)}}(r)+q$.
\end{proof}

\begin{lemma}\label{lem:add-last-weight}
Let $k\ge 2$, and let $0\le r<A_{k+1}$ be such that $A_k+r<A_{k+1}$. Then
\[
\mathsf m_{A^{(k)}}(A_k+r)=\mathsf m_{A^{(k)}}(r)+1.
\]
\end{lemma}

\begin{proof}
Write $r=qA_k+s$, with $0\le s<A_k$. Since $A_k+r=(q+1)A_k+s<A_{k+1}=aA_k+1$, we have $q+1\le a$. Applying Lemma~\ref{lem:splitting-lambda} to $r=qA_k+s$ and to $A_k+r=(q+1)A_k+s$, we obtain $\mathsf m_{A^{(k)}}(r)=\mathsf m_{A^{(k-1)}}(s)+q$ and $\mathsf m_{A^{(k)}}(A_k+r)=\mathsf m_{A^{(k-1)}}(s)+q+1$. The result follows.
\end{proof}

Now, for $1\le N\le A_{k+1}$, set
\[
\mathcal M_k(N):=
\left\{
\ell\in\{0,\ldots,N-1\}\ \middle|\
\begin{array}{l}
\mathsf m_{A^{(k)}}(\ell+A_j)\le \mathsf m_{A^{(k)}}(\ell)\\
\text{for every } j\in\{1,\ldots,k\}\text{ such that }\ell+A_j<N
\end{array}
\right\}.
\]
Thus, since $A^{(n-1)}=A$, Theorem~\ref{th:maxcrit_mA} shows that $\operatorname{Maximals}_{\preceq_S}\operatorname{Ap}(S)$ is precisely the set of elements $W(\ell)$ with $\ell\in\mathcal M_{n-1}(m)$.

We prove, by induction on $k$, that $|\mathcal M_k(N)|\le k$ for all $k\ge 1$ and all $1\le N\le A_{k+1}$.

\begin{lemma}\label{lem:full-level-maximals}
For every $k\ge 2$, $\mathcal M_{k-1}(A_k)=\{A_k-1-q\mid q=0,\ldots,k-2\}$. Moreover, $\mathsf m_{A^{(k-1)}}(A_k-1-q)=a+q(a-1)$ for $q=0,\ldots,k-2$.
\end{lemma}

\begin{proof}
The length formula follows from Lemma~\ref{lem:full-level-canonical-vectors}, applied with $k-1$ in place of $k$.

We prove the description of $\mathcal M_{k-1}(A_k)$ by induction on $k$. If $k=2$, then $A_1=1$ and $A_2=a+1$, so $\mathcal M_1(A_2)=\{A_2-1\}$.

Assume $k\ge 3$ and suppose the result known for $k-1$. Let $s\in\mathcal M_{k-1}(A_k)$. If $s+A_{k-1}<A_k$, then Lemma~\ref{lem:add-last-weight}, applied with $k-1$ in place of $k$, gives $\mathsf m_{A^{(k-1)}}(s+A_{k-1})=\mathsf m_{A^{(k-1)}}(s)+1$, contradicting the definition of $\mathcal M_{k-1}(A_k)$. Hence $s+A_{k-1}\ge A_k$, and so $s\ge A_k-A_{k-1}=(a-1)A_{k-1}+1$. Since $s\le A_k-1=aA_{k-1}$, either $s=A_k-1$, or $s=(a-1)A_{k-1}+r$ for some $1\le r\le A_{k-1}-1$.

In the second case, we claim that $r\in\mathcal M_{k-2}(A_{k-1})$. Indeed, let $j\in\{1,\ldots,k-2\}$ be such that $r+A_j<A_{k-1}$. Then $s+A_j<A_k$, and the defining condition for $s\in\mathcal M_{k-1}(A_k)$ gives $\mathsf m_{A^{(k-1)}}(s+A_j)\le \mathsf m_{A^{(k-1)}}(s)$. By Lemma~\ref{lem:splitting-lambda}, applied to $s=(a-1)A_{k-1}+r$ and to $s+A_j=(a-1)A_{k-1}+r+A_j$, this is equivalent to $\mathsf m_{A^{(k-2)}}(r+A_j)\le \mathsf m_{A^{(k-2)}}(r)$. Thus $r\in\mathcal M_{k-2}(A_{k-1})$. By the induction hypothesis, $r=A_{k-1}-1-p$ for some $p=0,\ldots,k-3$, and hence $s=aA_{k-1}-(p+1)=A_k-1-(p+1)$.

Conversely, let $s_q=A_k-1-q$ with $q=0,\ldots,k-2$. If $q=0$, then no $j$ satisfies $s_q+A_j<A_k$, so $s_q\in\mathcal M_{k-1}(A_k)$. Assume $q\ge 1$, and write $s_q=(a-1)A_{k-1}+r_q$, where $r_q=A_{k-1}-q$. Since $r_q=A_{k-1}-1-(q-1)$ and $q-1\in\{0,\ldots,k-3\}$, the induction hypothesis gives $r_q\in\mathcal M_{k-2}(A_{k-1})$.

Let $j\in\{1,\ldots,k-1\}$ be such that $s_q+A_j<A_k$. Then $j\ne k-1$, because $s_q+A_{k-1}=A_k-1-q+A_{k-1}\ge A_k$, as $q\le k-2<A_{k-1}$. Hence $j\le k-2$. Moreover, from $s_q+A_j=(a-1)A_{k-1}+r_q+A_j<A_k=aA_{k-1}+1$, we get $r_q+A_j\le A_{k-1}$. If $r_q+A_j<A_{k-1}$, the maximality of $r_q$ gives $\mathsf m_{A^{(k-2)}}(r_q+A_j)\le \mathsf m_{A^{(k-2)}}(r_q)$, and Lemma~\ref{lem:splitting-lambda} gives $\mathsf m_{A^{(k-1)}}(s_q+A_j)\le \mathsf m_{A^{(k-1)}}(s_q)$. If $r_q+A_j=A_{k-1}$, then $s_q+A_j=A_k-1$, and $\mathsf m_{A^{(k-1)}}(A_k-1)=a\le a+q(a-1)=\mathsf m_{A^{(k-1)}}(s_q)$. Thus $s_q\in\mathcal M_{k-1}(A_k)$.
\end{proof}

\begin{lemma}\label{lem:interband-elimination}
Let $k\ge 3$, and let $A_k<N\le A_{k+1}$. Write $N-1=cA_k+r$, with $1\le c\le a$ and $0\le r<A_k$. For $0\le q\le k-3$, set $\ell_q:=(c-1)A_k+A_k-1-q$. If $r\ge A_{q+2}-q-1$, then $\ell_q\notin\mathcal M_k(N)$.
\end{lemma}

\begin{proof}
Put $j:=q+2$. Since $q\le k-3$, we have $j\le k-1$. Moreover, $\ell_q+A_j=cA_k+A_{q+2}-q-1$. By assumption, $A_{q+2}-q-1\le r$, and so $\ell_q+A_j\le cA_k+r=N-1<N$. Thus $j$ is admissible in the definition of $\mathcal M_k(N)$.

By Lemmas~\ref{lem:splitting-lambda} and~\ref{lem:full-level-maximals}, $\mathsf m_{A^{(k)}}(\ell_q)=(c-1)+\mathsf m_{A^{(k-1)}}(A_k-1-q)=(c-1)+a+q(a-1)$. On the other hand, Lemma~\ref{lem:splitting-lambda} gives $\mathsf m_{A^{(k)}}(\ell_q+A_j)=c+\mathsf m_{A^{(k-1)}}(A_{q+2}-q-1)$.

Since $A_{q+2}-q-1<A_{q+2}$, no generator $A_i$ with $i\ge q+2$ can occur in a factorization of $A_{q+2}-q-1$. Hence $\mathsf m_{A^{(k-1)}}(A_{q+2}-q-1)=\mathsf m_{A^{(q+1)}}(A_{q+2}-q-1)$. By Lemma~\ref{lem:full-level-maximals}, applied with $q+2$ in place of $k$, we have $\mathsf m_{A^{(q+1)}}(A_{q+2}-q-1)=a+q(a-1)$. Thus $\mathsf m_{A^{(k)}}(\ell_q+A_j)=c+a+q(a-1)>\mathsf m_{A^{(k)}}(\ell_q)$.

Therefore the defining condition for $\ell_q\in\mathcal M_k(N)$ fails for the admissible index $j=q+2$.
\end{proof}

\begin{proposition}\label{prop:combinatorial-type-bound}
For every $k\ge 1$ and every $N$ with $1\le N\le A_{k+1}$, one has $|\mathcal M_k(N)|\le k$. Moreover, if $|\mathcal M_k(N)|=p\ge 2$, then $N-1\ge A_p-p+1$.
\end{proposition}

\begin{proof}
We prove both assertions simultaneously by induction on $k$. If $k=1$, then $A_1=1$ and $\mathsf m_{A^{(1)}}(x)=x$, so the inequality $\mathsf m_{A^{(1)}}(\ell+1)\le\mathsf m_{A^{(1)}}(\ell)$ is impossible whenever $\ell+1<N$. Hence $\mathcal M_1(N)=\{N-1\}$, and the result follows.

Assume $k\ge 2$ and suppose the result known for $k-1$. If $N\le A_k$, then the last coordinate in the canonical representative of every $x<N$ with respect to $A_1,\ldots,A_k$ is zero, and the condition involving $A_k$ is never admissible. Hence $\mathcal M_k(N)=\mathcal M_{k-1}(N)$, and both assertions follow from the induction hypothesis.

Assume now that $A_k<N\le A_{k+1}$, and write $N-1=cA_k+r$, with $1\le c\le a$ and $0\le r<A_k$. If $\ell\in\mathcal M_k(N)$, then $\ell+A_k\ge N$ by Lemma~\ref{lem:add-last-weight}. Thus every element of $\mathcal M_k(N)$ lies either in the upper band $\ell=cA_k+s$, $0\le s\le r$, or in the lower band $\ell=(c-1)A_k+s$, $r+1\le s\le A_k-1$.

For the upper band, Lemma~\ref{lem:splitting-lambda} shows that $cA_k+s\in\mathcal M_k(N)$ if and only if $s\in\mathcal M_{k-1}(r+1)$. Set $p:=|\mathcal M_{k-1}(r+1)|$. Since $r\in\mathcal M_{k-1}(r+1)$, we have $p\ge1$, and the upper band contributes exactly $p$ elements.

Let $\ell=(c-1)A_k+s$ be a lower-band element of $\mathcal M_k(N)$. If $j<k$ and $s+A_j<A_k$, then $\ell+A_j<N$, and Lemma~\ref{lem:splitting-lambda} gives $\mathsf m_{A^{(k-1)}}(s+A_j)\le \mathsf m_{A^{(k-1)}}(s)$. Hence $s\in\mathcal M_{k-1}(A_k)$, so Lemma~\ref{lem:full-level-maximals} gives $s=A_k-1-q$ for some $q=0,\ldots,k-2$.

If $p=1$, then the upper band contributes exactly one element. On the other hand, every lower-band element must satisfy
$
s\in \mathcal M_{k-1}(A_k)=\{A_k-1-q\mid q=0,\ldots,k-2\}
$
by Lemma~\ref{lem:full-level-maximals}, so the lower band has at most $k-1$ candidates. Therefore
$
|\mathcal M_k(N)|\le 1+(k-1)=k.
$
Assume $p\ge2$. By the induction hypothesis applied to $\mathcal M_{k-1}(r+1)$, we have $r\ge A_p-p+1$. Since $(A_{h+1}-h)-(A_h-h+1)=a^h-1>0$, the sequence $A_h-h+1$ is strictly increasing. It follows that $r\ge A_{q+2}-q-1$ for every $q\le p-2$. By Lemma~\ref{lem:interband-elimination}, these lower-band candidates are eliminated. Thus at most the candidates with $q=p-1,\ldots,k-2$ can survive, giving at most $k-p$ lower-band elements. Therefore $|\mathcal M_k(N)|\le p+(k-p)=k$.

It remains to prove the second assertion in the case $A_k<N\le A_{k+1}$. Let $|\mathcal M_k(N)|=u\ge2$. We already know that $u\le k$. Since $N>A_k$, we have $N-1\ge A_k$. As the sequence $A_h-h+1$ is increasing and $u\le k$, we get $A_u-u+1\le A_k-k+1\le A_k\le N-1$. This completes the induction.
\end{proof}

\begin{proof}[Proof of Theorem~\ref{thm:type-bound}]
By Proposition~\ref{prop:PF1}, $\operatorname{t}(S)=|\operatorname{Maximals}_{\preceq_S}\operatorname{Ap}(S)|$. Since $A^{(n-1)}=A$, Theorem~\ref{th:maxcrit_mA} gives $\operatorname{Maximals}_{\preceq_S}\operatorname{Ap}(S)=\{\,W(\ell)\mid \ell\in\mathcal M_{n-1}(m)\,\}$. Hence $\operatorname{t}(S)=|\mathcal M_{n-1}(m)|$. Since $m\le A_n$, Proposition~\ref{prop:combinatorial-type-bound}, applied with $k=n-1$ and $N=m$, gives $|\mathcal M_{n-1}(m)|\le n-1$. Therefore $\operatorname{t}(S)\le n-1$.
\end{proof}

\section{A family with arithmetic pseudo-Frobenius numbers}\label{sec:arith_pf_family}

Throughout this section we keep the standing hypotheses and notation of
Sections~\ref{sec:apery} and~\ref{sec:PF}. In particular,
$S=S_{a,b}(m)$, $d=(a-1)m+b$, and $n$ is the smallest positive integer
such that $A_n\ge m$.

We study the subfamily determined by
\[
m=cA_{n-1}+1
\]
for some $c\in\{1,\ldots,a\}$. Equivalently,
$m-1=cA_{n-1}$, and the canonical representative of $m-1$ in $R(a,n)$ is
$(0,\ldots,0,c)$.

This family belongs to the Collection CNS introduced in \cite{XIN}. However, the results on pseudo-Frobenius numbers and type obtained there apply only to more restrictive subfamilies and do not cover directly the specialization $m=cA_{n-1}+1$, $1\le c\le a$, considered here. We therefore give a direct proof that, for this subfamily, $\operatorname{PF}(S)$ is an arithmetic
progression of length $n-1$.

\begin{theorem}\label{th:arith_family_PF}
Assume that $m=cA_{n-1}+1$ for some $c\in\{1,\dots,a\}$, and set
$\alpha:=W(m-1)-m=d(m-1)+(c-1)m$. Then
\[
\operatorname{Maximals}_{\preceq_S}\operatorname{Ap}(S)
=
\{\,W(m-1-q)\mid q=0,\ldots,n-2\,\},
\]
and
\[
\operatorname{PF}(S)
=
\{\alpha,\alpha-b,\alpha-2b,\ldots,\alpha-(n-2)b\}.
\]
In particular, $\operatorname{PF}(S)$ is an arithmetic progression of
length $n-1$ with common difference $-b$, and $\operatorname{t}(S)=n-1$.
\end{theorem}

\begin{proof}
If $n=2$, then $A=\{A_1\}=\{1\}$ and
$W(i)=i(d+m)$ for $0\le i\le m-1$. Hence
$W(i)\preceq_S W(j)$ whenever $0\le i\le j\le m-1$, so
$W(m-1)$ is the unique maximal element of $\operatorname{Ap}(S)$. The conclusion follows from Proposition~\ref{prop:PF1}. Hence assume $n\ge 3$.

We first determine the maximal elements of the Ap\'ery poset. Let
$W(\ell)$ be maximal. Since $m-1=cA_{n-1}$, we have $r(m-1)=0$. By
Corollaries~\ref{cor:two_layers} and~\ref{cor:two_bands}, either
$\ell=m-1$, or $\ell=(c-1)A_{n-1}+r$ with $1\le r\le A_{n-1}-1$.

In the latter case, since $
\ell+A_{n-1}=cA_{n-1}+r=(m-1)+r\ge m$,
the index $j=n-1$ is not admissible in Theorem~\ref{th:maxcrit_mA}.
Thus, for every $j\in\{1,\ldots,n-2\}$ such that $r+A_j<A_{n-1}$, we have
$
\ell+A_j=(c-1)A_{n-1}+(r+A_j)<cA_{n-1}+1=m$,
and hence Theorem~\ref{th:maxcrit_mA} gives
$
\mathsf m_A(\ell+A_j)\le \mathsf m_A(\ell).
$
Applying Lemma~\ref{lem:splitting-lambda} to both sides, we obtain $
\mathsf m_{A^{(n-2)}}(r+A_j)\le \mathsf m_{A^{(n-2)}}(r)$.
Therefore $r\in \mathcal M_{n-2}(A_{n-1})$. Hence, by
Lemma~\ref{lem:full-level-maximals}, $r=A_{n-1}-1-p$ for some
$p=0,\ldots,n-3$. Thus $\ell=cA_{n-1}-1-p=m-1-(p+1)$.

Conversely, $W(m-1)$ is maximal by
Corollary~\ref{cor:mminus1_maximal}. Let $q\in\{1,\ldots,n-2\}$, and set
$\ell_q:=m-1-q=(c-1)A_{n-1}+r_q$, where $r_q=A_{n-1}-q$. By
Lemma~\ref{lem:full-level-maximals}, $r_q\in\mathcal M_{n-2}(A_{n-1})$.

We verify the criterion in Theorem~\ref{th:maxcrit_mA}. Let
$j\in\{1,\ldots,n-1\}$ and assume $\ell_q+A_j<m$. Then $j\ne n-1$,
because $\ell_q+A_{n-1}=m-1-q+A_{n-1}\ge m$, as
$q\le n-2<A_{n-1}$. Hence $j\le n-2$.

If $r_q+A_j<A_{n-1}$, then
$\mathsf m_{A^{(n-2)}}(r_q+A_j)\le\mathsf m_{A^{(n-2)}}(r_q)$, and
Lemma~\ref{lem:splitting-lambda} gives
$\mathsf m_A(\ell_q+A_j)\le\mathsf m_A(\ell_q)$. If
$r_q+A_j=A_{n-1}$, then $\ell_q+A_j=m-1$, and
$\mathsf m_A(m-1)=c\le c+q(a-1)=\mathsf m_A(\ell_q)$, where the last
equality follows from Lemmas~\ref{lem:splitting-lambda} and
\ref{lem:full-level-maximals}. Thus Theorem~\ref{th:maxcrit_mA} shows
that $W(\ell_q)=W(m-1-q)$ is maximal.

Therefore $\operatorname{Maximals}_{\preceq_S}\operatorname{Ap}(S)
=
\{\,W(m-1-q)\mid q=0,\ldots,n-2\,\}$.
By Proposition~\ref{prop:PF1},
$\operatorname{PF}(S)=\{\,W(m-1-q)-m\mid q=0,\ldots,n-2\,\}$.

It remains to identify this set explicitly. By
Lemmas~\ref{lem:splitting-lambda} and~\ref{lem:full-level-maximals},
$\mathsf m_A(m-1-q)=c+q(a-1)$ for $q=0,\ldots,n-2$. Hence
$W(m-2-q)-W(m-1-q)=-d+(a-1)m=-b$ for $q=0,\ldots,n-3$, because
$d=(a-1)m+b$.

Finally, $\mathsf m_A(m-1)=c$, so
$W(m-1)-m=d(m-1)+(c-1)m=\alpha$. Thus
$\operatorname{PF}(S)=
\{\alpha,\alpha-b,\alpha-2b,\ldots,\alpha-(n-2)b\}$, and the result
follows.
\end{proof}

Thus Theorem~\ref{th:arith_family_PF} shows that the upper bound $\operatorname{t}(S)\le n-1$ is sharp.

\begin{remark}
Assume that $n\ge 3$, and let $I_S$ be the toric ideal of $S$, with
$x_i$ corresponding to $s_{i-1}$ for $i=1,\ldots,n$. With the notation
introduced above, one has
$a s_{i-1}+s_j=a s_{j-1}+s_i$ for all $1\le i<j\le n-1$, because
$A_i=aA_{i-1}+1$ for every $i\ge 1$, with $A_0=0$. Moreover, for every
$1\le i\le n-1$, we have
$a s_{i-1}+(d+c+1-a)s_0=s_i+c s_{n-1}$.
Indeed,
\[
a s_{i-1}+(d+c+1-a)s_0
=
a(m+A_{i-1}d)+(d+c+1-a)m
=
(d+c+1)m+aA_{i-1}d,
\]
while
\[
s_i+c s_{n-1}
=
(m+A_i d)+c(m+A_{n-1}d)
=
(c+1)m+(A_i+cA_{n-1})d,
\]
and these two expressions coincide because $A_i=aA_{i-1}+1$ and
$m=cA_{n-1}+1$.

Therefore $I_S$ contains the $2\times2$ minors of the matrix
\[
X_c=
\begin{pmatrix}
x_1^a & x_2^a & \cdots & x_{n-1}^a & x_n^c\\
x_2   & x_3   & \cdots & x_n       & x_1^{\,d+1+c-a}
\end{pmatrix},
\]
that is, $I_2(X_c)\subseteq I_S$.

When $c=a$, equivalently $m=A_n$, the matrix above becomes
\[
X_a=
\begin{pmatrix}
x_1^a & x_2^a & \cdots & x_{n-1}^a & x_n^a\\
x_2   & x_3   & \cdots & x_n       & x_1^{\,d+1}
\end{pmatrix},
\]
which is exactly the determinantal presentation of the generalized
repunit case; see \cite[Theorem~1 and Corollary~2]{BrancoColacoOjeda21}.
Moreover, its minimal graded free resolution is given by the
Eagon--Northcott complex; see \cite{ColacoOjeda25}.

On the other hand, the determinantal criteria
\cite{GotoKienMatsuokaTruong18,KienMatsuoka20} involve hypotheses under
which $\operatorname{PF}(S)$ has the form
$\{\delta,2\delta,\ldots,(n-1)\delta\}$ for some
$\delta\in\mathbb Z\setminus S$; see
\cite[Theorem~1.2]{GotoKienMatsuokaTruong18} and
\cite[Theorem~2]{KienMatsuoka20}. Since
Theorem~\ref{th:arith_family_PF} gives instead
$\operatorname{PF}(S)=\{\alpha,\alpha-b,\ldots,\alpha-(n-2)b\}$, those
criteria do not apply directly to the whole family $m=cA_{n-1}+1$.

Nevertheless, the matrix $X_c$ provides a natural determinantal
candidate. This suggests the question of whether $I_S=I_2(X_c)$ for the
whole family. A plausible approach would be to adapt the Gr\"obner basis
and saturation arguments used in \cite{BrancoColacoOjeda21} for the
generalized repunit case.

This question, together with a detailed Gr\"obner basis analysis of $I_2(X_c)$, will be addressed in a subsequent work.
\end{remark}

\paragraph{\bf Acknowledgements.}
Microsoft Copilot was used to assist with English language editing and stylistic revision. The authors take full responsibility for the final content of the manuscript.



\begin{thebibliography}{99}

\bibitem{BDF97}
{Barucci, V., Dobbs, D. E. and Fontana, M.}:
\emph{Maximality properties in numerical semigroups and applications to one-dimensional analytically irreducible local domains},
Mem. Amer. Math. Soc. \textbf{125} (1997), no. 598.
\url{https://pubs.ams.org/ebooks/memo/0598}

\bibitem{BrancoColacoOjeda21}
{Branco, M. B., Cola\c{c}o, I. and Ojeda, I.}:
\emph{Minimal Systems of Binomial Generators for the Ideals of Certain Monomial Curves},
Mathematics \textbf{9} (24) (2021), 3204.
\url{https://doi.org/10.3390/math9243204}

\bibitem{BrancoColacoOjeda23}
{Branco, M. B., Cola\c{c}o, I. and Ojeda, I.}:
\emph{The Frobenius problem for generalized repunit numerical semigroups},
Mediterr. J. Math. \textbf{20} (2023), Art.~16.
\url{https://doi.org/10.1007/s00009-022-02233-w}

\bibitem{ColacoOjeda25}
{Cola\c{c}o, I. and Ojeda, I.}:
\emph{Minimal free resolution of generalized repunit algebras},
Commun. Algebra \textbf{53} (2) (2025), 909--916.
\url{https://doi.org/10.1080/00927872.2024.2394968}

\bibitem{numericalsgps}
{Delgado, M., Garc\'ia-S\'anchez, P. A. and Morais, J.}:
\emph{NumericalSgps, a package for numerical semigroups}, Version 1.3.0 (2022),
GAP package.
\url{https://gap-packages.github.io/numericalsgps}

\bibitem{GAP}
{The GAP Group}:
\emph{GAP -- Groups, Algorithms, and Programming}, Version 4.10.0 (2018).
\url{https://www.gap-system.org}

\bibitem{GimenezSrinivasan14}
{Gimenez, P. and Srinivasan, H.}:
\emph{A note on Gorenstein monomial curves},
Bull. Braz. Math. Soc. (N.S.) \textbf{45} (4) (2014), 671--678.
\url{https://doi.org/10.1007/s00574-014-0068-4}

\bibitem{GotoKienMatsuokaTruong18}
{Goto, S., Kien, D. V., Matsuoka, N. and Truong, H. L.}:
\emph{Pseudo-Frobenius numbers versus defining ideals in numerical semigroup rings},
J. Algebra \textbf{508} (2018), 1--15.
\url{https://doi.org/10.1016/j.jalgebra.2018.04.025}

\bibitem{JafariZarzuela18}
{Jafari, R. and Zarzuela Armengou, S.}:
\emph{Homogeneous numerical semigroups},
Semigroup Forum \textbf{97} (2018), 278--306.
\url{https://doi.org/10.1007/s00233-018-9941-6}

\bibitem{KienMatsuoka20}
{Kien, D. V. and Matsuoka, N.}:
\emph{Numerical Semigroup Rings of Maximal Embedding Dimension with Determinantal Defining Ideals},
in \emph{Numerical Semigroups},
Springer INdAM Ser. \textbf{40}, Springer, Cham, (2020), 185--196.
\url{https://doi.org/10.1007/978-3-030-40822-0_12}

\bibitem{XIN}
{Liu, F., Xin, G., Ye, S. and Yin, J.}:
\emph{On the Frobenius number and genus of a collection of semigroups generalizing repunit numerical semigroups},
Semigroup Forum (2025).
\url{https://doi.org/10.1007/s00233-025-10518-1}

\bibitem{ONeillPelayo18}
{O'Neill, C. and Pelayo, R.}:
\emph{Ap\'ery sets of shifted numerical monoids},
Adv. Appl. Math. \textbf{97} (2018), 27--35.
\url{https://doi.org/10.1016/j.aam.2018.01.005}

\bibitem{Ugo17}
{Ugolini, S.}:
\emph{On numerical semigroups closed with respect to the action of affine maps},
Publ. Math. Debrecen \textbf{90} (1--2) (2017), 149--167.
\url{https://doi.org/10.5486/PMD.2017.7500}

\bibitem{RoblesRosales18}
{Robles-P\'erez, A. M. and Rosales, J. C.}:
\emph{A combinatorial problem and numerical semigroups},
Ars Math. Contemp. \textbf{15} (2018), 323--336.
\url{https://doi.org/10.26493/1855-3974.989.d15}

\bibitem{libro}
{Rosales, J. C. and Garc\'ia-S\'anchez, P. A.}:
\emph{Numerical Semigroups},
Developments in Mathematics, Vol.~20, Springer, New York, (2009).
\url{https://doi.org/10.1007/978-1-4419-0160-6}

\bibitem{RosBraTorThabit}
{Rosales, J. C., Branco, M. B. and Torr\~ao, D.}:
\emph{The Frobenius problem for Thabit numerical semigroups},
J. Number Theory \textbf{155} (2015), 85--99.
\url{https://doi.org/10.1016/j.jnt.2015.03.006}

\bibitem{RosBraTorMersenne}
{Rosales, J. C., Branco, M. B. and Torr\~ao, D.}:
\emph{The Frobenius problem for Mersenne numerical semigroups},
Math. Z. \textbf{286} (2017), 741--749.
\url{https://doi.org/10.1007/s00209-016-1781-z}

\bibitem{RosBraTorRepunit}
{Rosales, J. C., Branco, M. B. and Torr\~ao, D.}:
\emph{The Frobenius problem for repunit numerical semigroups},
Ramanujan J. \textbf{40} (2016), 323--334.
\url{https://doi.org/10.1007/s11139-015-9719-3}

\end{thebibliography}
\end{document}